\documentclass[english]{amsart}
\usepackage{amstext}
\usepackage{amsthm}
\usepackage{stmaryrd}

\usepackage{hyperref}
\usepackage[capitalise]{cleveref}
\usepackage{comment}

\usepackage[]{todonotes}

\makeatletter

\numberwithin{equation}{section}
\numberwithin{figure}{section}
\theoremstyle{plain}
\newtheorem{thm}{\protect\theoremname}[section]
\theoremstyle{plain}
\newtheorem{lem}[thm]{\protect\lemmaname}
\theoremstyle{definition}
\newtheorem{defn}[thm]{\protect\definitionname}
\theoremstyle{plain}
\newtheorem{prop}[thm]{\protect\propositionname}
\theoremstyle{remark}
\newtheorem{rem}[thm]{\protect\remarkname}
\theoremstyle{remark}
\newtheorem*{claim*}{\protect\claimname}

\newtheorem{claimABC}{Claim}

\usepackage{tikz-cd, alex}

\newcommand{\KGL}{K_{\GL}}
\newcommand{\KSL}{K_{\SL}}

\newcommand{\Zimage}{\bar{Z}}

\newcommand{\Cimage}{\bar{C}(g)}
\newcommand{\Cim}{\bar{C}}

\makeatother

\usepackage{babel}
\providecommand{\claimname}{Claim}
\providecommand{\definitionname}{Definition}
\providecommand{\lemmaname}{Lemma}
\providecommand{\propositionname}{Proposition}
\providecommand{\remarkname}{Remark}
\providecommand{\theoremname}{Theorem}

\begin{document}
\title[The conjugation representation of $\GL_{2}$ and $\SL_{2}$]{The conjugation representation of $\GL_{2}$ and $\SL_{2}$ over finite
local rings}
\author{Nariel Monteiro and Alexander Stasinski}
\begin{abstract}
The conjugation representation of a finite group $G$ is the complex permutation module defined by the action of $G$ on itself by conjugation. 
Addressing a problem raised by Hain motivated by the study of a Hecke action on iterated Shimura integrals,
Tiep proved that for $G=\SL_{2}(\Z/p^{r})$,
where $r\geq1$ and $p\geq5$ is a prime, any irreducible representation
of $G$ that is trivial on the centre of $G$ is contained in the
conjugation representation. Moreover, 
Tiep asked whether this can be generalised to $p=2$ or $3$. We answer the  Hain--Tiep question in the affirmative
and also prove analogous statements for $\SL_{2}$ and $\GL_{2}$
over any finite local principal ideal ring with residue field of odd
characteristic. 
\end{abstract}

\address{Department of Mathematical Sciences, Durham University, Durham, DH1
3LE, UK}
\email{alexander.stasinski@durham.ac.uk}

\address{Department of Mathematics, University of California, Santa Cruz, CA 95064, USA}
\email{namontei@ucsc.edu}

\maketitle

\tableofcontents

\section{Introduction}
For a finite group $G$, the \emph{conjugation representation} is
the complex permutation representation induced by the action of $G$
on itself by conjugation. The character of this representations is the \emph{conjugation character} of $G$.

A basic problem is to determine the irreducible constituents of the
conjugation character. It is easy to see that any such constituent
must be trivial on the centre of $G$. Two early studies of the
problem were Frame \cite{Frame-conj_rep} and Roth \cite{Roth-conj_rep},
the latter of which conjectured that every irreducible representation
of $G$ that is trivial on the centre is contained in the conjugation
character. Counter-examples were, however, quickly given by Frame
in his MathSciNet review of \cite{Roth-conj_rep} and by Formanek
in \cite{Formanek-conj_rep}. The problem was subsequently taken up
by Passman \cite{Passman-conj_rep} who proved, among other things,
that every irreducible character of a symmetric group (on at least
three letters) is contained in the conjugation character.
In \cite{Heide-Zalesski}, Heide and Zalesski proved that the simple
alternating groups and the sporadic groups satisfy the same property.
It was then proved by Heide, Saxl, Tiep and Zalesski \cite{Heide-Saxl-Tiep-Zalesski}
that any finite simple group of Lie type $G$ satisfies the same property,
unless $G=\text{PSU}_{n}(\F_{q})$ and $n\geq3$ is coprime to $2(q+1)$
(where exactly one irreducible character fails to be a constituent).
It is therefore known that every finite non-abelian simple group other
than these $\text{PSU}_{n}(\F_{q})$ counter-examples,
has the property that every irreducible character appears in the conjugation
character. 

In recent work motivated by a problem raised by Hain \cite{hain2023} arising
in the study of a Hecke action on iterated Shimura integrals, Tiep
\cite{tiep2023conjugation} proved that for $G=\SL_{2}(\Z/p^{r})$,
where $r\geq1$ and $p\geq5$ is a prime, every irreducible character
of $G$ that is trivial on the centre of $G$ is contained in the
conjugation character. Tiep also raised the question of whether 
this can be generalised to $p=2$ or $3$. 

In the present paper we answer the Hain--Tiep question in the affirmative (\cref{thm:Main for SL_2}). 
Let $\cO$ be a DVR with finite residue field $\F_{q}$ of characteristic
$p$ and maximal ideal $\mfp$ and let $\cO_{i}=\cO/\mfp^{i}$ for
any integer $i\geq1$. 
We moreover show that for $G=\SL_{2}(\cO_r)$ and $G=\GL_{2}(\cO_r)$,
with $p>2$, any irreducible character that is trivial on the centre of $G$ is contained in the conjugation character (\cref{thm: main thm when p odd}).

Our method of proof is quite different from that of Tiep \cite{tiep2023conjugation}
for $p\geq5$, even though both approaches show that
for every irreducible character $\rho$ of $G$ that is trivial on
the centre, there is an element $g\in G$ such that $\rho$ restricted
to the centraliser $C_{G}(g)$ contains the trivial character. Moreover, both approaches exploit the fact that most irreducible characters of $G$ are induced from certain subgroups and that Mackey's formula implies that the required condition can be checked on the inducing representation.
However, the elements $g$ used by Tiep in \cite[(7.7)]{tiep2023conjugation}
do not have any useful analogues for $p=2$ or $3$, as for $p=3$,
$g$ is scalar modulo $p$, hence does not satisfy \cite[Lemma~8]{tiep2023conjugation}
and for $p=2$, the element $t(1-p)$ used to define $g$ for $g\geq5$
does not exist in general. We have therefore taken a different approach
to the elements $g$ and in Section~\ref{sec:A-family-of} introduce
the elements $g_{\hat{\beta}}$ attached to a regular matrix $\hat{\beta}$,
which represents the Clifford orbit of an irreducible character of $G$. Therefore,
in our approach, the elements $g_{\hat{\beta}}$ used will vary with
the representations whose restriction to $C_{G}(g)$ we want to contain
the trivial character.

Our proof does not require the full construction of characters of
$\GL_{2}(\cO_{r})$ (see \cite{Alex_smooth_reps_GL2, Stasinski-Stevens}) or $\SL_{2}(\cO_{r})$ (see  \cite{Tanaka1,kutzkothesis,Shalika} for $p\neq2$ and  \cite{Tanaka2,Nobs-SL2,Nobs-Wolfart-SL2} for the case $\cO_r=\Z/2^r$).
In particular, our proofs are logically independent of any of these
constructions, except that we use a Heisenberg lift construction from
\cite{Stasinski-Stevens} in Section~\ref{subsec:The-case-when-r-odd-v-less-l'}.
On the other hand, we have had to develop several new representation
theoretic constructions for $\SL_{2}(\cO_{r})$ and $\GL_{2}(\cO_{r})$, namely Proposition~\ref{prop:There exist extns to U^l'-vK^l by formula and every extn is of this form},
the results in Section~\ref{subsec:The-case-when-r-odd-v-bigger-0}
leading up to and including Proposition~\ref{prop:sigma induced from CJ}
and everything in Sections~\ref{sec:The-case-GL2,v=00003D2,p=00003D2} and \ref{sec:The case SL_2, p is 2}.
In addition to this, what we need from the representation theory of
$\SL_{2}(\cO_{r})$ and $\GL_{2}(\cO_{r})$ are well-known properties of
the characters $\psi_{\beta}$, their stabilisers and basic facts
from Clifford theory and regular representations, all of which are
summarised in Section~\ref{sec:Basics-of-the-chars}.

\subsection{Notation}
We introduce some notation that will be in place throughout the rest of the paper. This is only the notation needed for the remainder of this introduction and significant further notation will be introduced later in the text.

By a slight abuse of notation, we will
use $\mfp$ to denote the maximal ideal in each of the finite local
rings $\cO_{i}$. Let $\bfG$ be either $\GL_{2}$ or $\SL_{2}$,
considered as group schemes over $\cO$. From now on and throughout
the paper, fix an arbitrary integer $r\geq2$ and let $G=\bfG(\cO_{r})$.
Since we will mostly be working with $G$, but occasionally also with
$\mathbf{G}(\cO_{i})$ for $i\neq r$, our convention will be to denote
the latter groups by $G_{i}$. For any $1\leq i\leq r$ we have a
surjective homomorphism $\rho_{i}:G\rightarrow G_{i}$ induced by
the map $\cO_{r}\rightarrow\cO_{i}$ and also the obvious analogous
map $\rho_{i}:\M_{2}(\cO_{r})\rightarrow\M_{2}(\cO_{i})$ between
rings of matrices. Let $K^{i}$ be the $i$-th kernel of $\rho_{i}$
in $G$. Let $l=\floor{r/2}$ and $l'=\ceil{r/2}$ so that $r=2l=2l'$
if $r$ is even, and $r=2l-1=2l'+1$ if $r$ is odd. Let $Z$ denote
the centre of $G$.

\subsection{Informal outline of method of proof}
We now informally outline part of our proof in more detail and 

Assume first that $p\neq2$. This case works uniformly in $p$, but the
case where $p=2$ is rather different (more on this below). In this
case, all our main results also hold equally for $\SL_{2}(\cO_{r})$
as for $\GL_{2}(\cO_{r})$, which is why we use the notation $G$
for either of these groups. We will outline the proof of Theorem~\ref{thm:p odd - r odd - v < l'},
which is the case where $r$ is odd and $v<l'$ ($v$ is defined below)
as this also showcases the ideas used in the simpler case where $r$
is even. To follow the argument, it may be helpful to refer to the
following diagrams.

$$
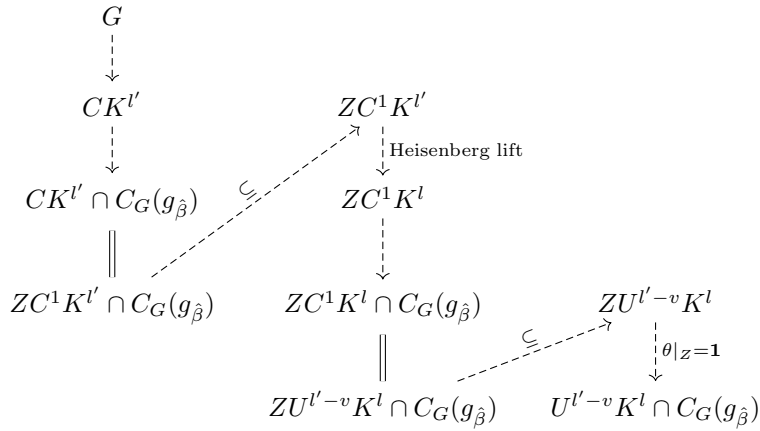
\begin{tikzcd}[column sep=0.4cm] 
G\arrow[dashrightarrow]{d}  &   & &\\
CK^{l'}\arrow[dashrightarrow]{d} 
&  ZC^1K^{l'} \arrow[dashrightarrow, "\text{Heisenberg lift}"]{d} 
& \\
CK^{l'}\cap C_G(g_{\hat{\beta}}) \arrow[equal]{d}
& ZC^1K^{l}\arrow[dashrightarrow]{d} 
&   \\
ZC^1K^{l'}\cap C_G(g_{\hat{\beta}})\arrow[dashrightarrow, "\subseteq", sloped]{uur}
& ZC^1K^{l} \cap C_G(g_{\hat{\beta}}) \arrow[equal]{d}
& ZU^{l'-v}K^l \arrow[dashrightarrow,"\theta|_Z = \mathbf{1}"]{d} \\
& 
ZU^{l'-v}K^l\cap C_G(g_{\hat{\beta}}) \arrow[dashrightarrow, "\subseteq", sloped]{ur} 
& U^{l'-v}K^l\cap C_G(g_{\hat{\beta}})
\end{tikzcd}
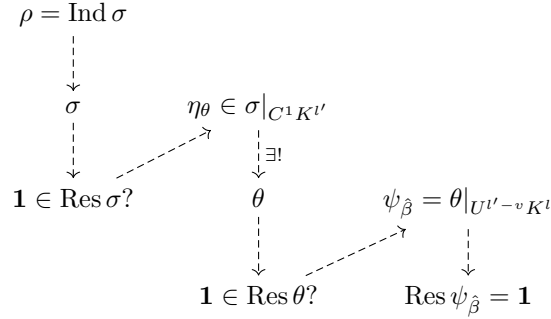
$$$$
\begin{tikzcd}[column sep=0.4cm] 
\rho=\Ind\sigma\arrow[dashrightarrow]{d}  &   & &\\
\sigma\arrow[dashrightarrow]{d} 
&  \eta_{\theta}\in \sigma|_{C^{1}K^{l'}}\arrow[dashrightarrow,"\exists !"]{d} 
& \\
\mathbf{1}\in\Res\sigma ?\arrow[dashrightarrow]{ur} 
& \theta\arrow[dashrightarrow]{d} 
& \psi_{\hat{\beta}}=\theta|_{U^{l'-v}K^{l}} \arrow[dashrightarrow]{d}  \\
& \mathbf{1}\in\Res\theta ? \arrow[dashrightarrow]{ur} 
& \Res\psi_{\hat{\beta}}=\mathbf{1}
\end{tikzcd}
$$The dashed arrows indicate the logical order of the steps and the
whole argument is a series of reduction steps, moving from $G$ to
successively smaller groups, until a known statement is reached. We
now explain these steps. All non-obvious statements are proved in
the paper. Let $\rho\in\Irr(G)$ be a character that is trivial on
$Z$. Assume that $\rho$ is regular, that is, $\rho$ contains a
character $\psi_{\beta}\in\Irr(K^{l})$, where $\beta$ is a regular
matrix in $\M_{2}(\cO_{l'})$. We may assume that
$\beta=\begin{pmatrix}0 & \lambda\\
\lambda^{-1}\Delta & 0
\end{pmatrix}$, where $\lambda$ is a unit (the trace of $\beta$ is necessarily $0$, as $\rho$ is trivial on $Z$). Then $\rho$ is induced from a character
$\sigma\in\Irr(CK^{l'})$, where $C:=C_{G}(\hat{\beta})$ and $\hat{\beta}\in\M_{2}(\cO_{r})$
is a lift of $\beta$ and we may take $\hat{\beta}=\begin{pmatrix}0 & \hat{\lambda}\\
\hat{\lambda}^{-1}\hat{\Delta} & 0
\end{pmatrix}$, where $\hat{\lambda}$ and $\hat{\Delta}$ are lifts of $\lambda$
and $\Delta$, respectively. We want to show that there exists an
element $g\in G$ such that $\rho$ restricted to $C_{G}(g)$ contains
the trivial character. By an elementary lemma, this will follow if
$\sigma$ restricted to $CK^{l'}\cap C_{G}(g)$ contains the trivial
character. We introduce the element
\[
g_{\hat{\beta}}=\begin{pmatrix}1 & \hat{\lambda}\\
-\hat{\lambda}^{-1}\hat{\Delta} & 1-\hat{\Delta}
\end{pmatrix}\in\SL_{2}(\cO_{r})
\]
attached to $\hat{\beta}$. Note that we use the elements $g_{\hat{\beta}}$ both for $ \SL_{2}(\cO_{r})$ and for $G=\GL_{2}(\cO_{r})$. By Lemma~\ref{lem:intersection of centralisers in ZUK},
\[
C_{G}(\hat{\beta})K^{l'}\cap C_{G}(g_{\hat{\beta}})=ZU^{l'-v}K^{l'}\cap C_{G}(g_{\hat{\beta}}),
\]
where $U^{i}=\begin{pmatrix}1 & \mfp^{i}\\
0 & 1
\end{pmatrix}\subset G$ and $v$ is the valuation of $\Delta\in\cO_{l'}$. Thus, when $v<l'$
(i.e., $\Delta\neq0$ in $\cO_{l'}$), we can show that
\[
ZC^{1}K^{l'}\cap C_{G}(g_{\hat{\beta}})=ZU^{l'-v}K^{l'}\cap C_{G}(g_{\hat{\beta}}),
\]
where $C^{1}=C\cap K^{1}$. (Note that this cannot possibly hold when
$v=l'$ and $U:=U^{0}\not\subseteq ZC^{1}K^{l'}$.) The dashed arrow
with the label ``$\subseteq$'' indicates that because of this group
containment, in order to prove that $\sigma$ restricted to $CK^{l'}\cap C_{G}(g)$
contains the trivial character for some $g\in G$, it is sufficient
to prove that the restriction of $\sigma$ to $ZC^{1}K^{l'}\subseteq CK^{l'}$
contains an irreducible constituent that contains the trivial character
when restricted to $ZC^{1}K^{l'}\cap C_{G}(g_{\hat{\beta}})$, for
some $g_{\hat{\beta}}$.

At this stage, we use a general result from \cite{Stasinski-Stevens}
which implies that there exists a ``Heisenberg lift'' of any $ZC^{1}K^{l'}$-stable
character $\theta\in\Irr(ZC^{1}K^{l}\mid\psi_{\beta})$, that is,
a \emph{unique} character $\eta_{\theta}\in\Irr(ZC^{1}K^{l'})$ lying
above $\theta$ and that every irreducible character of $ZC^{1}K^{l'}$
is of this form. (Note that if $v=l'$, we would not be able to work
with $ZC^{1}K^{l'}$ and Heisenberg lifts  from $ZUK^{l}$
to $ZUK^{l'}$ do not exist because the former group is not normal in the latter.
This is the reason why the case $v=l'$ requires a different approach.)
We may thus choose $\theta$ and $\eta_{\theta}$ such that $\eta_{\theta}$
is an irreducible constituent of $\sigma|_{ZC^{1}K^{l'}}$. Since
$\Ind_{ZC^{1}K^{l}}^{ZC^{1}K^{l'}}\theta$ is a multiple of $\eta_{\theta}$,
an elementary result 
implies that it is sufficient to prove that $\theta$ is trivial on
$ZC^{1}K^{l}\cap C_{G}(g_{\hat{\beta}})$, for some $g_{\hat{\beta}}$
($\theta$ is linear so its restriction only has one irreducible constituent).

Another application of Lemma~\ref{lem:intersection of centralisers in ZUK}
yields
\[
ZC^{1}K^{l}\cap C_{G}(g_{\hat{\beta}})=ZU^{l'-v}K^{l}\cap C_{G}(g_{\hat{\beta}}),
\]
and in analogy with a previous step, this means that it is enough
to prove that $\theta$ restricts to a character of $ZU^{l'-v}K^{l}\subseteq ZC^{1}K^{l}$
that is trivial on $ZU^{l'-v}K^{l}\cap C_{G}(g_{\hat{\beta}})$, for
some $g_{\hat{\beta}}$. But since $\rho$ was assumed to be trivial
on $Z$, it is necessary that $\theta$ is trivial on $Z$ and by
Proposition~\ref{prop:There exist extns to U^l'-vK^l by formula and every extn is of this form},
the restriction of $\theta$ to $U^{l'-v}K^{l}$ is of the form $\psi_{\hat{\beta}}$
for some lift $\hat{\beta}$ of the form $\begin{pmatrix}0 & \hat{\lambda}\\
\hat{\lambda}^{-1}\hat{\Delta} & 0
\end{pmatrix}$. The fact that $\psi_{\hat{\beta}}$ is trivial on $U^{l'-v}K^{l}\cap C_{G}(g_{\hat{\beta}})$
is the content of Lemma~\ref{lem: psi_hat-beta res to U^l'-vK^l cap C is the trivial char}
and this finishes the proof of Theorem~\ref{thm:p odd - r odd - v < l'}.

We remark that in the proof of the case when $r$ is even, the middle
columns of the above diagrams can be omitted, as $l=l'$ and there
is no Heisenberg lift needed.

We also note that the case when $r$ is odd and $v=l'$ requires a
somewhat different, but analogous, approach, for the reasons explained
above. The idea here is to show that every irreducible character of
$G$ containing $\psi_{\beta}$ can be induced from a linear character
of $C_{G}(\hat{\beta})J$, where $\hat{\beta}$ is a lift as above
and $J=(B\cap K^{l'})K^{l}$, where $B$ is the upper triangular subgroup
of $G$.

To prove our main theorem when $p\neq2$ (Theorem~\ref{thm: main thm when p odd}),
most of the work is for the regular characters, that is, the cases
outlined above. In addition, one has to also handle the non-regular
characters, which are non-twist primitive, that is, factor through
$G_{r-1}$ after multiplying by a linear character. This works by
an inductive argument, starting with $r=1$.

\bigskip

Consider now the case when $p=2$. When $G=\GL_{2}(\cO_{r})$, all
of our results for regular characters also hold for $p=2$, except
Lemmas~\ref{lem:reducing to twist-primitive when p odd}
and Lemma~\ref{lem:intersection CJ cap C(g) is ZU^1K^l cap C(g)}
(in the case $v=l')$. For $G=\GL_{2}(\cO_{r})$, $v=l'$ and residue
field $\cO_{1}=\F_{2}$ (see Section~\ref{sec:The-case-GL2,v=00003D2,p=00003D2}),
we are however able to carry out a Heisenberg lift construction as
in the case where $p\neq2$ and $v<l'$, based on the fact that in
this case, the group $Z^{l'}UK^{l}$ is actually normal in $UK^{l'}$.
These results for $\GL_{2}(\cO_{r})$ with $p=2$, are used, together
with several other ingredients, to solve the Hain--Tiep problem for $\SL_{2}(\Z/2^{r})$
in the final section. One main problem in the case $\SL_{2}(\cO_{r})$
with $p=2$ is that the centre does not in general map surjectively
onto the centre of $\SL_{2}(\cO_{r-1})$. Much greater care therefore
has to be taken when proving that representations that factor through
$\SL_{2}(\cO_{i})$, for some $1\leq i<r$, are contained in the conjugation
representation.

We note that when $p=2$, several of our proofs are dependent on the
fact that the residue field is $\F_{2}$ and the absolute ramification index is $1$. In particular, we do not know whether the analogue of the Hain--Tiep
problem holds when $\cO$ has residue field $\F_4$ or is ramified.

\bigskip\noindent

\subsection*{Acknowledgement}
Parts of this paper were established while the first-named author was visiting Durham University. We express our gratitude for the hospitality experienced there. The first-named author was supported by the National Science Foundation MPS-Ascend Postdoctoral Research Fellowship under Grant No. 2213166.

\section{Preliminary lemmas}

In this section, we collect some lemmas for general finite groups,
which relate the conjugation character, induced representations, and
restrictions to centralisers.

We will use $\mathbf{1}$ to denote the trivial character of a given finite group. It will always be clear from the context which group we are talking about.

The following lemma is well-known (see \cite[Lemma~1.5]{Roth-conj_rep}
or \cite[Lemma~3]{tiep2023conjugation}). It will be used implicitly several times in the following.
\begin{lem}
\label{lem:in conj rep if trivial on C}
Let $G$ be a finite group and $\chi\in\Irr(G)$. Then $\chi$ is
contained in the conjugation representation of $G$ if and only if
there exists a $g\in G$ such that $\chi|_{C_{G}(g)}$ contains the
trivial character.
\end{lem}

\begin{lem}
\label{lem: rho trivial on Z iff sigma trivial on Z}Let
$G$ be a finite group with centre $Z$ and $H\leq G$ a subgroup.
Let $\sigma\in\Irr(H)$. Then $\Ind_{H}^{G}\sigma$ is
trivial on $Z$ if and only if $\sigma$ is trivial on $H\cap Z$.
\end{lem}
\begin{proof}
By the Mackey intertwining number formula,
\begin{align*}
\langle(\Ind_{H}^{G}\sigma)|_{Z},\mathbf{1}\rangle & =\langle\Ind_{H}^{G}\sigma,\Ind_{Z}^{G}\mathbf{1}\rangle=\sum_{h\in H\backslash G/Z}\langle\sigma|_{H\cap Z},\mathbf{1}\rangle\\
 & =|H\backslash G/Z|\cdot\langle\sigma|_{H\cap Z},\mathbf{1}\rangle,
\end{align*}
whence the result.
\end{proof}
\begin{lem}
\label{lem: is contained in the conj rep iff sigma }
Let $N$, $H$ and $C$ be three subgroups of a finite group $G$
such that $N\leq H$ and let $\theta\in\Irr(N)$. Then the following
holds.
\begin{enumerate}
\item \label{enu: i)}If $\theta|_{N\cap C}$ contains $\mathbf{1}$,
then $(\Ind_{N}^{H}\theta)|_{H\cap C}$ contains $\mathbf{1}$. 
\item \label{enu: ii)}Suppose that $\Ind_{N}^{H}\theta$ is a multiple
of some $\eta\in\Irr(H)$. If $\theta|_{N\cap C}$ contains $\mathbf{1}$,
then $\eta|_{H\cap C}$ contains $\mathbf{1}$.
\end{enumerate}
\end{lem}

\begin{proof}
By the Mackey intertwining number formula, 
\begin{align*}
\langle(\Ind_{N}^{H}\theta)|_{H\cap C},\mathbf{1}\rangle & =\langle\Ind_{N}^{H}\theta,\Ind_{H\cap C}^{H}\mathbf{1}\rangle=\sum_{h\in N\backslash H/(H\cap C)}\langle\theta|_{N\cap\leftexp{h}{C}},\mathbf{1}\rangle\\
 & \geq\langle\theta|_{N\cap C},\mathbf{1}\rangle.
\end{align*}
Thus, if the right-hand side is positive, then so is the left-hand
side.

For the second part, note that by the first part, if $\theta|_{N\cap C}$
contains $\mathbf{1}$, then some irreducible constituent
of $\Ind_{N}^{H}\theta$ must contain $\mathbf{1}$ when
restricted to $H\cap C$, so if $\Ind_{N}^{H}\theta$ is a multiple
of a single irreducible constituent $\eta$, then $\eta|_{H\cap C}$
must contain~$\mathbf{1}$.
\end{proof}
\begin{lem}
\label{lem:normal subgroup trivial on centraliser}Let
$G$ be a finite group with centre $Z$ and let $N\trianglelefteqslant G$
be a normal subgroup. Let $Z'$ be a subgroup of $Z$.
Let $\sigma \in \Irr(N)$ and assume that $\rho\in\Irr(G\mid \sigma)$ 
is trivial on $Z'$. Assume that there exists a $g\in N$ such that $\rho$ restricted
to $C_{G}(g)$ contains the trivial character. Then there exists an $h\in N$ such that
$\sigma$ restricted to $C_{N}(h)$ contains the trivial character.
\end{lem}

\begin{proof}
Let $C:=C_G(g)$. By hypothesis, $\langle\rho,\Ind_{C}^{G}\mathbf{1}\rangle\neq0$
and therefore
\[
0\neq\langle\Ind_{N}^{G}\sigma,\Ind_{C}^{G}\mathbf{1}\rangle=\sum_{x\in N\backslash G/C}\langle\sigma|_{N\cap\leftexp{x}{C}},\mathbf{1}\rangle,
\]
so there exists an $x\in G$ such that 
\[
\langle\sigma|_{N\cap\leftexp{x}{C}},\mathbf{1}\rangle\neq0.
\]
Note that $N\cap\leftexp{x}{C}=N\cap C_{G}(\leftexp{x}{g})=C_{N}(\leftexp{x}{g})$
and as $g\in N$ and $N$ is normal in $G$, we have $\leftexp{x}{g}\in N$. Letting $h=\leftexp{x}{g}$,
we conclude that $\sigma$ restricted to $C_{N}(h)$ contains the
trivial character.
\end{proof}

\begin{lem}
\label{lem:deduce the main property for a normal subgroup}
Let $G$ be a finite group with centre $Z$ and let $N\trianglelefteqslant G$
be a normal subgroup. Assume that for every $\rho\in\Irr(G)$ that
is trivial on $Z$, there exists a $g\in N$ such that $\rho$ restricted
to $C_{G}(g)$ contains the trivial character. Then, for any $\sigma\in\Irr(N)$
that is trivial on $Z(N)$, there exists an $h\in N$ such that
$\sigma$ restricted to $C_{N}(h)$ contains the trivial character.
\end{lem}
\begin{proof}
Assume that  $\sigma\in\Irr(N)$ is trivial on $Z(N)$. As $N\cap Z \subseteq Z(N)$, \cref{lem: rho trivial on Z iff sigma trivial on Z} implies that there exists a $\rho\in \Irr(G\mid \sigma)$ that is trivial on $Z$. By hypothesis, there exists a $g \in N$ such that $\rho$ restricted
to $C_{G}(g)$ contains the trivial character. The conclusion now follows from \cref{lem:normal subgroup trivial on centraliser}.
\end{proof}

\section{Basics of the characters of $\mathrm{GL}_2(\mathcal{O}_r)$ and $\mathrm{SL}_2(\mathcal{O}_r)$}
\label{sec:Basics-of-the-chars}

In this section we summarise the construction of regular characters
of the groups $G$ via Clifford theory, following the approach laid
out in \cite{Alex_smooth_reps_GL2}. We only state the results that
we will need in the following namely, the description of regular characters
as induced from stabilisers of characters $\psi_{\beta}$, which is
much weaker than the full construction of regular characters as given
in \cite{Stasinski-Stevens}. 

For any integer $i$ such that $r\geq i\geq1$, let $K^{i}$ be the
kernel of the homomorphism $\rho_{i}:G\rightarrow G_{i}$. For $i>r$
we define $K^{i}=\{1\}$. A fundamental fact that we will use tacitly
in the following is the commutator relation
\[
[K^{i},K^{j}]\subseteq K^{i+j},
\]
which can be proved by direct matrix computations. In particular,
it shows that $K^{l}$ is abelian and that it centralises the group
$K^{l'}$.

Let $\mfg=\Lie(\bfG)$, so that
\[
\mfg=\begin{cases}
\M_{2}=\gl_{2} & \text{if }\bfG=\GL_{2},\\
\sl_{2} & \text{if }\bfG=\SL_{2}.
\end{cases}
\]
Assume from now on that $i\geq r/2$. Then the group $K^{i}$ is abelian
and we have a $G$-equivariant isomorphism
\[
K^{i}\longiso\mfg(\cO_{r-i}),\qquad1+\pi^{i}X\longmapsto\rho_{r-i}(X).
\]
From now on, and throughout the rest of the paper, let $\psi:\cO_{r}\rightarrow\C^{\times}$
be a fixed additive character that does not factor through $\cO_{r-1}$.
When $\bfG=\GL_{2}$ or $\bfG=\SL_{2}$ and $p\neq2$, the trace form
$(x,y)\mapsto\Tr(xy)$ on $\mfg(\cO_{r-i})$ is non-degenerate, so we have $G$-equivariant isomorphism
\begin{equation}
\mfg(\cO_{r-i})\longiso\Irr(K^{i}),\qquad\beta\longmapsto\psi_{\beta},\label{eq: psi_betas in the good cases}
\end{equation}
where $G$ acts on $\mfg(\cO_{r-i})$ via its quotient $G_{r-i}$.
$$\psi_{\beta}(I+\pi^{i}X):=\psi(\Tr(\hat{\beta}\pi^{i}X))$$
and $\hat{\beta}\in\mfg(\cO_{r})$ is an arbitrary lift of $\beta$
(under the map $\rho_{r-i}$). When $\bfG=\SL_{2}$ and $p=2$, restriction
from $\Irr(\KGL^{i})$ to $\Irr(\KSL^{i})$ is a surjective homomorphism
whose kernel consists of those $\psi_{\beta}$ where $\beta$ is a
scalar matrix. This induces an isomorphism
\begin{equation}
\M_{2}(\cO_{r-i})/\cO_{r-i}I\longiso\Irr(\KSL^{i}),\qquad[\beta]\longmapsto\psi_{[\beta]}:=\Res_{\KSL^{i}}^{\KGL^{i}}\psi_{\beta},\label{eq: psi_betas in the bad case}
\end{equation}
where $[\beta]$ denotes the coset of $\beta$ modulo scalar matrices.

Now let $i=l$ and consider $\psi_{\beta}\in\Irr(K^{l})$. When $\bfG=\GL_{2}$
or $\bfG=\SL_{2}$ and $p\neq2$, the stabiliser $\Stab_{G}(\psi_{\beta}):=\{g\in G\mid\psi_{\beta}(g^{-1}kg)=\psi_{\beta}(k),\ \forall k\in K^{l}\}$
is given by 
\begin{equation}
\Stab_{G}(\psi_{\beta})=C_{G}(\hat{\beta})K^{l'},\label{eq: stabiliser in the good cases}
\end{equation}
for any lift $\hat{\beta}\in\mfg(\cO_{r})$ of $\beta$ (see \cite[Corollary~3.7]{Hill_regular}
for the case $\GL_{2}$; the same argument works for $\SL_{2}$ with
$p\neq2$). When $\bfG=\SL_{2}$ and $p=2$, the stabiliser $\Stab_{G}(\psi_{[\beta]})$
is in general strictly bigger than $C_{G}(\hat{\beta})K^{l'}$ and
was determined for $\cO_{r}$ of characteristic $2$ in \cite{Hasa-Stasinski}
and more generally in \cite{M-Singla}. We will determine $\Stab_{G}(\psi_{[\beta]})$
for $\SL_{2}(\Z/2^{r})$ later in this paper, where it is needed.
\begin{defn}
A representation (or character) $\rho$ of $G$ (with $r\geq2$) is
called \emph{primitive} if it does not factor through $G_{r-1}$ and
is called \emph{twist primitive} if $\chi\otimes\rho$ (or $\chi\rho$)
is primitive for all one-dimensional representations (linear characters)
$\chi$ of $G$.
\end{defn}
Note that twist primitive characters are called primitive in \cite{Alex_smooth_reps_GL2}.
By Clifford's theorem, the restriction of any $\rho\in\Irr(G)$ to
$K^{l}$ contains a full orbit of characters $\psi_{\beta}$ under
the action of $G$. Thus, by the above, $\rho$ uniquely determines
a conjugation orbit $\Omega_{l'}(\rho)\subseteq\mfg(\cO_{l'})$ (or
$\M_{2}(\cO_{l'})/\cO_{l'}I$ when $\bfG=\SL_{2}$ and $p=2$). Similarly,
restricting to $K^{r-1}$ determines an orbit $\Omega_{r-1}(\rho)\subseteq\mfg(\F_{q})$
(or $\M_{2}(\F_{q})/\F_{q}I$ when $\bfG=\SL_{2}$ and $p=2$). Recall
that a matrix $A\in\M_{n}(K)$, where $K$ is a field, is called regular
if it is conjugate to its companion matrix. It is well-known that
$A$ is regular if and only if the centraliser $C_{\M_{n}(K)}(A)$
consists of polynomials in $A$, that is, if $C_{\M_{n}(K)}(A)=\{f(A)\mid f(X)\in K[X]\}$.
In particular, if $A\in\M_{2}(K)$ is regular, then $C_{\M_{n}(K)}(A)=\{aI+bA\mid a,b\in K\}$.
By changing the field $K$ to the ring $\cO_{r}$ one can define regularity
for matrices in $\M_{n}(\cO_{r})$ in an analogous way and by a theorem
of Hill \cite[Theorem~3.6]{Hill_regular} a matrix $A\in\M_{n}(\cO_{r})$
is regular iff its image $\rho_{1}(A)\in\M_{n}(\F_{q})$ is regular
iff $C_{\M_{n}(\cO_{r})}(A)=\{aI+bA\mid a,b\in\cO_{r}\}$ iff $A$
is $\GL_{n}(\cO_{r})$-conjugate to its companion matrix.
\begin{defn}
A character $\rho\in\Irr(G)$ is called \emph{regular} if its orbit
$\Omega_{r-1}(\rho)$ consists of matrices that are regular in $\M_{2}(\F_{q})$
(or cosets of regular matrices modulo $\F_{q}I$ when $\bfG=\SL_{2}$
and $p=2$).
\end{defn}

If $H$ is a finite group, $H'\subseteq H$ is a subgroup, $\rho\in\Irr(H)$
and $\rho'\in\Irr(H')$, we will write $\rho\in\Irr(H\mid\rho')$
to indicate that $\rho$ lies above $\rho'$ (in the sense of Clifford
theory), that is, that $\rho|_{H'}$ contains $\rho'$ as an irreducible
constituent.
\begin{lem}
\label{lem:constr_regular_chars_for_GL2_or_SL2}Let
$\rho\in\Irr(G)$. 
\begin{enumerate}
\item \label{enu: constr lemma i)}If $\rho$ is twist
primitive, then $\rho$ is regular. 
\item Assume that $\bfG=\SL_{2}$. If $\rho$ is primitive, then $\rho$
is regular.
\item Assume that $\bfG=\GL_{2}$ or $p\neq2$. If $\rho$ is regular, then
there is an element $\beta\in\mfg(\cO_{l'})$ such that $\beta$ is
regular in $\M_{2}(\cO_{l'})$ and a $\sigma\in\Irr(C_{G}(\hat{\beta})K^{l'}\mid\psi_{\beta})$,
for any choice of lift $\hat{\beta}\in\M_{2}(\cO_{r})$ of $\beta$,
such that
\[
\rho=\Ind_{C_{G}(\hat{\beta})K^{l'}}^{G}\sigma.
\]
\item Assume that $\bfG=\SL_{2}$ and $p=2$. If $\rho$ is regular, then
there is a regular element $\beta\in\M_{2}(\cO_{l'})$ and a $\sigma\in\Irr(\Stab_{G}(\psi_{[\beta]})\mid\psi_{[\beta]})$,
such that
\[
\rho=\Ind_{\Stab_{G}(\psi_{[\beta]})}^{G}\sigma.
\]
\end{enumerate}
Moreover, the above statements also hold when $\beta$ is replaced
by any other element in its orbit $\Omega_{l'}(\rho)$.
\end{lem}

\begin{proof}
Assume first that $\bfG=\GL_{2}$ or $p\neq2$ and let $\beta\in\M_{2}(\F_{q})$
be such that $\beta\in\Omega_{r-1}(\rho)$. If $\beta$ is scalar
and $\bfG=\GL_{2}$, then $\rho$ is not twist primitive (see \cite[Section~2.1]{Alex_smooth_reps_GL2})
and if $\beta$ is scalar and $\bfG=\SL_{2}$, then $\beta=0$ so
$\rho$ is not even primitive. Thus, in either case, if $\rho$ is
twist primitive, then $\beta$ is non-scalar and as the scalar matrices
are the only non-regular elements in $\M_{2}(\F_{q})$, this implies
that $\beta$ is regular and hence that $\rho$ is regular. Now assume
that $\bfG=\SL_{2}$ and $p=2$ and let $[\beta]\in\Omega_{r-1}(\rho)$.
If $\beta$ is scalar, then clearly $[\beta]=[0]$ and $\psi_{[0]}=\mathbf{1}$,
so $\rho$ is not primitive. Thus, if $\rho$ is twist primitive,
then $\beta$ is non-scalar, so as in the previous case, $\rho$ is
regular.

The second statement follows from the proof of the first statement,
as we have actually showed that when $\bfG=\SL_{2}$ and $\rho$ is
primitive, then $\Omega_{r-1}(\rho)$ contains a non-scalar element
$\beta$ (or a non-zero class $[\beta]$ when $p=2$) and a non-scalar
element in $\M_{2}(\F_{q})$ is regular.

The third statement follows immediately from standard Clifford theory
(see, e.g., \cite[(6.11)]{Isaacs}) together with (\ref{eq: psi_betas in the good cases})
and (\ref{eq: stabiliser in the good cases}).
The fourth statement follows from the same Clifford theory result
together with (\ref{eq: psi_betas in the bad case}).
The last assertion follows by conjugation.
\end{proof}
By definition, a regular element in $\M_{2}(\cO_{i})$, $i\geq1$,
is $\GL_{2}(\cO_{i})$-conjugate to its companion matrix, which is
of the form $\begin{pmatrix}0 & 1\\
\Delta & \tau
\end{pmatrix}$, for $\Delta,\tau\in\cO_{i}$. Thus it is also $\GL_{2}(\cO_{i})$-conjugate
to $\begin{pmatrix}\tau & 1\\
\Delta & 0
\end{pmatrix}$, which is the form we will use henceforth. As $\GL_{2}(\cO_{i})=\SL_{2}(\cO_{i})\begin{pmatrix}\cO_{i}^{\times} & 0\\
0 & 1
\end{pmatrix}$, every regular element in $\M_{2}(\cO_{i})$ is $\SL_{2}(\cO_{i})$-conjugate
to an element of the form $\begin{pmatrix}\tau & \lambda\\
\lambda^{-1}\Delta & 0
\end{pmatrix}$, where $\lambda\in\cO_{i}^{\times}$. In particular, when $i=l'$,
every regular element in $\M_{2}(\cO_{l'})$ is $\SL_{2}(\cO_{l'})$-conjugate
to an element

\[
\beta:=\begin{pmatrix}\tau & \lambda\\
\lambda^{-1}\Delta & 0
\end{pmatrix},\qquad\Delta,\tau\in\cO_{l'},\ \lambda\in\cO_{l'}^{\times}.
\]

If $G=\SL_{2}(\cO_{r})$ and $p\neq2$, then by (\ref{eq: psi_betas in the good cases})
the $\psi_{\beta}\in\Irr(K^{l})$ that are contained in regular representations
 are parametrised by $\beta=\begin{pmatrix}0 & \lambda\\
\lambda^{-1}\Delta & 0
\end{pmatrix}\in\sl_{2}(\cO_{l'})$, that is, with $\tau=0$. The following simple lemma shows that the
analogous fact holds when $G_{}=\GL_{2}(\cO_{r})$, for any $p$,
as long as we consider regular representations that are trivial on
the centre $Z$ of $G$. 
\begin{lem}
\label{lem: tau can be taken as 0}Let
$G=\GL_{2}(\cO_{r})$ and let $\rho\in\Irr(G)$ be a regular representation
that is trivial on $K^{l}\cap Z$. Then $\rho$ contains a character
$\psi_{\beta}$, where $\beta=\begin{pmatrix}0 & 1\\
\Delta & 0
\end{pmatrix}\in\M_{2}(\cO_{l'})$ (that is, $\Tr(\beta)=0$). 
\end{lem}

\begin{proof}
As was discussed in the beginning of this section, $\rho$ being regular
implies that it contains a $\psi_{\beta}$ with $\beta=\begin{pmatrix}\tau & 1\\
\Delta & 0
\end{pmatrix}\in\M_{2}(\cO_{l'})$. Since $\rho$ is trivial on $Z$, it is trivial when restricted
to $K^{l}\cap Z$. Thus, by Clifford's theorem, 
\[
\rho|_{K^{l}\cap Z}=\rho(1)\psi_{\beta}|_{K^{l}\cap Z}=\rho(1)\cdot \mathbf{1},
\]
hence $\psi_{\beta}|_{K^{l}\cap Z}=\mathbf{1}$. We
have $K^{l}\cap Z=\left\{ \begin{pmatrix}1+\pi^{l}a & 0\\
0 & 1+\pi^{l}a
\end{pmatrix}\mid a\in\cO_{r}\right\} $, so
\[
\psi_{\beta}\begin{pmatrix}1+\pi^{l}a & 0\\
0 & 1+\pi^{l}a
\end{pmatrix}=\psi(\hat{\tau}\pi^{l}a)=1,
\]
for all $a$ (where $\hat{\tau}\in\cO_{r}$ is an arbitrary lift of
$\tau$). Since $\psi(\mfp^{r-1})\neq1$, this implies that $\hat{\tau}\in\mfp^{l'}$,
hence $\tau=0$.
\end{proof}

\section{Extensions of $\psi_{\beta}$ to $U^{l'-v}K^{l}$}

In this section we show that the characters $\psi_{\beta}$ introduced
in the previous section extend ``by the same formula'' to a certain
group $U^{l'-v}K^{l}$ which will play a central role in the rest
of the paper.

In the previous section we saw that every regular element in $\M_{2}(\cO_{l'})$
is $\SL_{2}(\cO_{l'})$-conjugate to an element

\[
\beta=\begin{pmatrix}\tau & \lambda\\
\lambda^{-1}\Delta & 0
\end{pmatrix},\qquad\Delta,\tau\in\cO_{l'},\ \lambda\in\cO_{l'}^{\times}.
\]
For $\beta$ as above, we will write 
\[
\hat{\beta}=\begin{pmatrix}\hat{\tau} & \hat{\lambda}\\
\hat{\lambda}^{-1}\hat{\Delta} & 0
\end{pmatrix}\in\M_{2}(\cO_{r}),
\]
where $\hat{\Delta},\hat{\tau},\hat{\lambda}\in\cO_{r}$ are some
chosen lifts of $\Delta$, $\tau$ and $\lambda$, respectively. (Note
that not all lifts of $\beta$ are of this form, as the bottom right
entry of $\hat{\beta}$ is $0$.) Let 
\[
v=\min\{\val(\Delta),\val(\tau)\},
\]
where $\val(x)$ for $x\in\cO_{l'}$ is defined to be the largest
integer in $[0,l']$ such that $x\in\mfp^{\val(x)}$. Thus in particular,
\[
\hat{\Delta}\in\mfp^{v}\quad\text{and}\quad\hat{\tau}\in\mfp^{v}.
\]
Moreover, $v=l'$ if and only if $\Delta=\tau=0$. 

For any $i\geq0$, define the subgroup
\[
U^{i}=\begin{pmatrix}1 & \mfp^{i}\\
0 & 1
\end{pmatrix}
\]
 of $G$. We note that 
\begin{equation}
U^{l'-v}\subseteq C_{G}(\hat{\beta})K^{l'}.\label{eq:U^l'-v is in CK^l'}
\end{equation}
Indeed, $\rho_{l'}(C_{G}(\hat{\beta}))=C_{G_{l'}}(\beta)=\left\{ \begin{pmatrix}x+\tau y & \lambda y\\
\lambda^{-1}\Delta y & x
\end{pmatrix}\mid x,y\in\cO_{l'}\right\} \cap G_{l'}\supseteq\rho_{l'}(U^{l'-v})$, since for $y\in\mfp^{l'-v}$, we have $\Delta y\equiv0\mod{\mfp}^{l'}$
and $\tau y\equiv0\mod{\mfp}^{l'}$. The same argument shows that
\begin{equation}
U^{l-v}\subseteq C_{G}(\hat{\beta})K^{l}.\label{eq:U^l-v is in CK^l}
\end{equation}

\begin{prop}
\label{prop:There exist extns to U^l'-vK^l by formula and every extn is of this form}
Let $\beta$ and $\hat{\beta}$ be as above. The function $\psi_{\hat{\beta}}:U^{l'-v}K^{l}\rightarrow\C^{\times}$
given by
\[
\psi_{\hat{\beta}}(x)=\psi(\Tr(\hat{\beta}(x-1))),\qquad\text{for }x\in U^{l'-v}K^{l},
\]
is a linear character. If $\hat{\beta}'=\begin{pmatrix}\hat{\tau}' & \hat{\lambda}'\\
(\hat{\lambda}')^{-1}\hat{\Delta}' & 0
\end{pmatrix}\in\M_{2}(\cO_{r})$ is another lift of $\beta$, then $\psi_{\hat{\beta}}=\psi_{\hat{\beta}'}$
if and only if $\hat{\lambda}^{-1}\hat{\Delta}\equiv(\hat{\lambda}')^{-1}\hat{\Delta}'\mod{\mfp}^{l+v}$.
Moreover, every element of $\Irr(U^{l'-v}K^{l}\mid\psi_{\beta})$
is of the form $\psi_{\hat{\beta}}$.
\end{prop}

\begin{proof}
Any element of $U^{l'-v}K^{l}$ has the form 
\[
\begin{pmatrix}1+a\pi^{l} & b\pi^{l'-v}\\
c\pi^{l} & 1+d\pi^{l}
\end{pmatrix}
\]
for some $a,b,c,d\in\cO_{r}$ and thus
\begin{align*}
\psi_{\hat{\beta}}\left(\begin{pmatrix}1+a\pi^{l} & b\pi^{l'-v}\\
c\pi^{l} & 1+d\pi^{l}
\end{pmatrix}\right) & =\psi\left(\Tr\left(\hat{\beta}\begin{pmatrix}a\pi^{l} & b\pi^{l'-v}\\
c\pi^{l} & d\pi^{l}
\end{pmatrix}\right)\right)\\
 & =\psi(\hat{\tau}a\pi^{l}+\hat{\lambda}c\pi^{l}+b\hat{\lambda}^{-1}\hat{\Delta}\pi^{l'-v}).
\end{align*}
Note that $\hat{\tau}a\pi^{l}$ and $\hat{\lambda}c\pi^{l}$ are independent
of the choices of lifts $\hat{\tau}$ and $\hat{\lambda}$ of $\tau$
and $\lambda$, respectively, and therefore $\psi_{\hat{\beta}}$
only depends on the image of $\hat{\lambda}^{-1}\hat{\Delta}$ mod
$\mfp^{l+v}$. 

We first show that $\psi_{\hat{\beta}}$ is a linear character of
$U^{l'-v}K^{l}$. For any $a',b',c',d'\in\cO_{r}$,
\begin{align*}
 & \psi_{\hat{\beta}}\left(\begin{pmatrix}1+a\pi^{l} & b\pi^{l'-v}\\
c\pi^{l} & 1+d\pi^{l}
\end{pmatrix}\begin{pmatrix}1+a'\pi^{l} & b'\pi^{l'-v}\\
c'\pi^{l} & 1+d'\pi^{l}
\end{pmatrix}\right)\\
 & =\psi_{\hat{\beta}}\left(\begin{pmatrix}1+(a+a')\pi^{l}+bc'\pi^{r-v} & ((b+b')+(ab'+bd')\pi^{l})\pi^{l'-v}\\
(c+c')\pi^{l} & \cdots
\end{pmatrix}\right)\\
 & =\psi(\hat{\tau}(a+a')\pi^{l}+\hat{\tau}bc'\pi^{r-v}+\hat{\lambda}(c+c')\pi^{l}+((b+b')+(ab'+bd')\pi^{l})\hat{\lambda}^{-1}\hat{\Delta}\pi^{l'-v})\\
 & =\psi(\hat{\tau}(a+a')\pi^{l}+\hat{\lambda}(c+c')\pi^{l}+(b+b')\hat{\lambda}^{-1}\hat{\Delta}\pi^{l'-v})\quad(\text{since }\hat{\Delta},\hat{\tau}\in\mfp^{v})\\
 & =\psi(\hat{\tau}a\pi^{l}+\hat{\lambda}c\pi^{l}+b\hat{\lambda}^{-1}\hat{\Delta}\pi^{l'-v})\psi(\hat{\tau}a'\pi^{l}+\hat{\lambda}c'\pi^{l}+b'\hat{\lambda}^{-1}\hat{\Delta}\pi^{l'-v}).
\end{align*}
Thus $\psi_{\hat{\beta}}$ is a homomorphism.

We have already seen that $\psi_{\hat{\beta}}$ only depends on $\hat{\lambda}^{-1}\hat{\Delta}$
mod $\mfp^{l+v}$. Conversely, if $\psi_{\hat{\beta}}=\psi_{\hat{\beta}'}$,
then
\[
\psi(\hat{\tau}a\pi^{l}+\hat{\lambda}c\pi^{l}+b\hat{\lambda}^{-1}\hat{\Delta}\pi^{l'-v})=\psi(\hat{\tau}'a\pi^{l}+\hat{\lambda}'c\pi^{l}+b(\hat{\lambda}')^{-1}\hat{\Delta}'\pi^{l'-v}),
\]
for all $a,b,c\in\cO_{r}$ ($\hat{\tau}a\pi^{l}=\hat{\tau}'a\pi^{l}$
and $\hat{\lambda}c\pi^{l}=\hat{\lambda}'c\pi^{l}$ in $\cO_{r}$),
so $\psi(b\pi^{l'-v}(\hat{\lambda}^{-1}\hat{\Delta}-(\hat{\lambda}')^{-1}\hat{\Delta}'))=1$
for all $b$; hence $\pi^{l'-v}(\hat{\lambda}^{-1}\hat{\Delta}-(\hat{\lambda}')^{-1}\hat{\Delta}')=0$
(as $\psi$ is non-trivial on $\mfp^{r-1}$), that is, 
\[
\hat{\lambda}^{-1}\hat{\Delta}-(\hat{\lambda}')^{-1}\hat{\Delta}'\in\mfp^{r-(l'-v)}=\mfp^{l+v}.
\]
We now show that every element of $\Irr(U^{l'-v}K^{l}\mid\psi_{\beta})$
is of the form $\psi_{\hat{\beta}}$. By definition, every $\psi_{\hat{\beta}}$
is an extension of $\psi_{\beta}$. The number of elements of $\Irr(U^{l'-v}K^{l}\mid\psi_{\beta})$
is $|U^{l'-v}K^{l}/K^{l}|=|U^{l'-v}/U^{l}|=q^{v+1}$. But we have
shown that the number of distinct characters $\psi_{\hat{\beta}}$
is equal to the number of distinct lifts of $\lambda^{-1}\Delta\in\cO_{l'}$
to $\cO_{l+v}$, which also equals $q^{v+1}$. Thus the $\psi_{\hat{\beta}}$
exhaust $\Irr(U^{l'-v}K^{l}\mid\psi_{\beta})$.
\end{proof}

\section{\label{sec:A-family-of}Special group elements
and intersections with their centralisers}

As expressed in Lemma~\ref{lem:in conj rep if trivial on C},
a character $\rho\in\Irr(G)$ is contained in the conjugation character
if and only if $\rho$ restricted to some centraliser $C_{G}(g)$,
$g\in G$, contains the trivial character. In this section we will define a family of elements that satisfy
this, for the various irreducible characters of $G$ with $p\neq2$, as well as for many characters of $\SL_{2}(\cO_r)$ when $p=2$. 

For any $\hat{\beta}=\begin{pmatrix}0 & \hat{\lambda}\\
\hat{\lambda}^{-1}\hat{\Delta} & 0
\end{pmatrix}\in\M_{2}(\cO_{r})$, $\hat{\lambda}\in\cO_{r}^{\times}$, we define the element
\[
g_{\hat{\beta}}=\begin{pmatrix}1 & \hat{\lambda}\\
-\hat{\lambda}^{-1}\hat{\Delta} & 1-\hat{\Delta}
\end{pmatrix}\in\SL_{2}(\cO_{r}).
\]

\begin{rem}
Even though $g_{\hat{\beta}}$ is closely related to $\Big(\begin{smallmatrix}1 & \hat{\lambda}\\
\hat{\lambda}^{-1}\hat{\Delta} & 1+\hat{\Delta}
\end{smallmatrix}\Big)\in\SL_{2}(\cO_{r})$, this latter element would not work for our purposes as the proof
of the crucial Lemma~\ref{lem: psi_hat-beta res to U^l'-vK^l cap C is the trivial char}
would fail.
\end{rem}

Since $g_{\hat{\beta}}$ is always regular as an element of $\GL_{2}(\cO_{r})$,
its centraliser in $G$ is
\begin{equation}
C_{G}(g_{\hat{\beta}})=\cO_{r}[g_{\hat{\beta}}]\cap G=\left\{ \begin{pmatrix}x & y\hat{\lambda}\\
-y\hat{\lambda}^{-1}\hat{\Delta} & x-y\hat{\Delta}
\end{pmatrix}\mid x,y\in\cO_{r}\right\} \cap G.\label{eq:centraliser of g_beta}
\end{equation}
For an element $x\in\cO_{r}$ and an integer $r\geq i\geq1$, we will
write $x_{i}:=\rho_{i}(x)\in\cO_{i}$ for the image of $x \mod \mfp^{i}$.
This notation is extended to matrices in $\M_{2}(\cO_{r})$ in the
obvious way. Let 
\[
w:=\val(\hat{\Delta}_{l}).
\]
Note that if $v<l'$, then $w=\val(\hat{\Delta}_{l})=\val(\Delta)=v$,
but if $v=l'$ (i.e., $\Delta=0$), then we can have $w=l'$ or $l$
(e.g., if $\hat{\Delta}=0$, then $w=l$ but if $\hat{\Delta}=\pi^{l'}$,
then $w=l'$). Note also that we always have $l'-v \leq l-w$, so  $U^{l-w} \subseteq  U^{l'-v}$ and therefore (\ref{eq:U^l-v is in CK^l}) implies that
\begin{equation}
\label{eq:Ul-w inside is in CK^l}
U^{l-w}\subseteq C_{G}(\hat{\beta})K^{l}.
\end{equation}

\begin{lem}
\label{lem:intersection of centralisers in ZUK}Let
$\beta=\begin{pmatrix}0 & \lambda\\
\lambda^{-1}\Delta & 0
\end{pmatrix}\in\sl_{2}(\cO_{l'}),\lambda\in\cO_{l'}^{\times}$ and let $\hat{\beta}=\begin{pmatrix}0 & \hat{\lambda}\\
\hat{\lambda}^{-1}\hat{\Delta} & 0
\end{pmatrix}\in\sl_{2}(\cO_{r})$ be a lift of $\beta$. Then 
\begin{enumerate}
\item $C_{G}(\hat{\beta})K^{l'}\cap C_{G}(g_{\hat{\beta}})=ZU^{l'-v}K^{l'}\cap C_{G}(g_{\hat{\beta}})$;
\item $C_{G}(\hat{\beta})K^{l}\cap C_{G}(g_{\hat{\beta}})=ZU^{l-w}K^{l}\cap C_{G}(g_{\hat{\beta}})$.
\end{enumerate}
\end{lem}

\begin{proof}
Let $h\in C_{G}(\hat{\beta})K^{l'}\cap C_{G}(g_{\hat{\beta}})$. Then
\begin{align*}
h_{l'} & \in C_{G_{l'}}(\beta)\cap C_{G_{l'}}\Big(\begin{smallmatrix}1 & \lambda\\
-\lambda^{-1}\Delta & 1-\Delta
\end{smallmatrix}\Big)\\
 & \subseteq\left\{ \begin{pmatrix}a & b\lambda\\
b\lambda^{-1}\Delta & a
\end{pmatrix}\mid a,b\in\cO_{l'}\right\} \cap\left\{ \begin{pmatrix}x & y\lambda\\
-y\lambda^{-1}\Delta & x-y\Delta
\end{pmatrix}\mid x,y\in\cO_{l'}\right\}.
\end{align*}
A matrix in this intersection is of the form
\[
\begin{pmatrix}a & b\lambda\\
b\lambda^{-1}\Delta & a
\end{pmatrix}=\begin{pmatrix}x & y\lambda\\
-y\lambda^{-1}\Delta & x-y\Delta
\end{pmatrix},
\]
which implies that $y\Delta=0$ in $\cO_{l'}$, that is, $y\in\mfp^{l'-v}.$
Hence
\[
h_{l'}\in\left\{ \begin{pmatrix}x & y\\
0 & x
\end{pmatrix}\mid x\in\cO_{l'},\,y\in\mfp^{l'-v}\right\} \cap G_{l'},
\]
and therefore $h\in ZU^{l'-v}K^{l'}$. We have now established that
\[
C_{G}(\hat{\beta})K^{l'}\cap C_{G}(g_{\hat{\beta}})\subseteq ZU^{l'-v}K^{l'},
\]
so trivially also $C_{G}(\hat{\beta})K^{l'}\cap C_{G}(g_{\hat{\beta}})\subseteq ZU^{l'-v}K^{l'}\cap C_{G}(g_{\hat{\beta}})$.
On the other hand, $ZU^{l'-v}K^{l'}\subseteq C_{G}(\hat{\beta})K^{l'}$
by (\ref{eq:U^l'-v is in CK^l'}), so $ZU^{l'-v}K^{l'}\cap C_{G}(g_{\hat{\beta}})\subseteq C_{G}(\hat{\beta})K^{l'}\cap C_{G}(g_{\hat{\beta}})$
and therefore 
\[
C_{G}(\hat{\beta})K^{l'}\cap C_{G}(g_{\hat{\beta}})=ZU^{l'-v}K^{l'}\cap C_{G}(g_{\hat{\beta}}).
\]

Next, let $h\in C_{G}(\hat{\beta})K^{l}\cap C_{G}(g_{\hat{\beta}})$.
Then 
\begin{align*}
h_{l} & \in C_{G_{l}}(\hat{\beta}_{l})\cap C_{G_{l}}\Big(\begin{smallmatrix}1 & \hat{\lambda}_{l}\\
-\hat{\lambda}_{l}^{-1}\hat{\Delta}_{l} & 1-\hat{\Delta}_{l}
\end{smallmatrix}\Big)\\
 & \subseteq\left\{ \begin{pmatrix}a & b\hat{\lambda}_{l}\\
b\hat{\Delta}_{l} & a
\end{pmatrix}\mid a,b\in\cO_{l}\right\} \cap\left\{ \begin{pmatrix}x & y\hat{\lambda}_{l}\\
-y\hat{\lambda}_{l}^{-1}\hat{\Delta}_{l} & x-y\hat{\Delta}_{l}
\end{pmatrix}\mid x,y\in\cO_{l}\right\}.
\end{align*}
A matrix in this intersection is of the form
\[
\begin{pmatrix}a & b\hat{\lambda}_{l}\\
b\hat{\Delta}_{l} & a
\end{pmatrix}=\begin{pmatrix}x & y\hat{\lambda}_{l}\\
-y\hat{\lambda}_{l}^{-1}\hat{\Delta}_{l} & x-y\hat{\Delta}_{l}
\end{pmatrix},
\]
which implies that 
\[
y\hat{\Delta}_{l}=0.
\]
If $w=l$ (i.e., $\hat{\Delta}_{l}=0$), then trivially $y\in\mfp^{l-w}=\mfp^{0}=\cO_{l}$.
Otherwise $w<l$ and therefore $y\in\mfp^{l-w}$. Thus, in either
case, $y\in\mfp^{l-w}$ and
\[
h_{l}\in\left\{ \begin{pmatrix}x & y\\
0 & x
\end{pmatrix}\mid x\in\cO_{l},\,y\in\mfp^{l-w}\right\} \cap G_{l},
\]
so that $h\in ZU^{l-w}K^{l}$. Thus 
\[
C_{G}(\hat{\beta})K^{l}\cap C_{G}(g_{\hat{\beta}})\subseteq ZU^{l-w}K^{l}
\]
 and by the same argument as above, but using (\ref{eq:U^l-v is in CK^l}),
we obtain 
\[
C_{G}(\hat{\beta})K^{l}\cap C_{G}(g_{\hat{\beta}})=ZU^{l-w}K^{l}\cap C_{G}(g_{\hat{\beta}}).
\]
\end{proof}
Even though the following lemma has an extremely simple proof, it
embodies one of the key properties (together with Lemma~\ref{lem:intersection of centralisers in ZUK})
that we need about the elements $g_{\hat{\beta}}$ and will be used
repeatedly in the proof of our main theorems.
\begin{lem}
\label{lem: psi_hat-beta res to U^l'-vK^l cap C is the trivial char}Let
$\beta=\begin{pmatrix}0 & \lambda\\
\lambda^{-1}\Delta & 0
\end{pmatrix}\in\sl_{2}(\cO_{l'}),\lambda\in\cO_{l'}^{\times}$. For every lift $\hat{\beta}=\begin{pmatrix}0 & \hat{\lambda}\\
\hat{\lambda}^{-1}\hat{\Delta} & 0
\end{pmatrix}\in\sl_{2}(\cO_{r})$ of $\beta$, the restriction of $\psi_{\hat{\beta}}$ to $U^{l'-v}K^{l}\cap C_{G}(g_{\hat{\beta}})$
is the trivial character.
\end{lem}

\begin{proof}
Let $h\in U^{l'-v}K^{l}\cap C_{G}(g_{\hat{\beta}})$. By (\ref{eq:centraliser of g_beta}),
$h$ is of the form
\[
h=\begin{pmatrix}x & y\hat{\lambda}\\
-y\hat{\lambda}^{-1}\hat{\Delta} & x-y\hat{\Delta}
\end{pmatrix},\qquad x\in1+\mfp^{l},\ y\in\mfp^{l'-v}
\]
and thus
\begin{align*}
\psi_{\hat{\beta}}(h) & =\psi\left(\Tr\begin{pmatrix}0 & \hat{\lambda}\\
\hat{\lambda}^{-1}\hat{\Delta} & 0
\end{pmatrix}\begin{pmatrix}x-1 & y\hat{\lambda}\\
-y\hat{\lambda}^{-1}\hat{\Delta} & x-1-y\hat{\Delta}
\end{pmatrix}\right)\\
 & =\psi(-y\hat{\Delta}+\hat{\Delta}y)=\psi(0)=1.
\end{align*}
\end{proof}

\section{The case when $p$ is odd}

The goal of this section is to prove that our main theorem for $p\neq2$,
that is, that when $p\neq2$, any $\rho\in\Irr(G)$ that is trivial
on $Z$ is contained in the conjugation character of $G$. The non-twist
primitive characters are relatively easy to take care of in this case,
so the main focus is regular representations and we show that if $G_{}=\GL_{2}(\cO_{r})$
or $G_{}=\SL_{2}(\cO_{r})$ and $p\neq2$, then any regular representation
that is trivial on $Z$ is contained in the conjugation character
of $G$. Our approach necessitates three cases: The case when $r=2l$
is, as always in the representation theory of $G$, easier; the case
when $r\neq2l$ and $v<l'$ is dealt with via a `Heisenberg lift' from
$U^{l'-v}K^{l}$ to $U^{l'-v}K^{l'}$ and the case $v=l'$ requires
a different approach to constructing the regular characters. These
cases are dealt with in their respective subsections. 

Our first step is to prove results that enable us to handle non-twist
primitive characters. If $\rho\in\Irr(G)$ is not twist primitive,
then, by definition, there is a linear character $\chi$ of $G$ such
that $\chi\rho$ is trivial on $K^{r-1}$. Since $\chi$ is trivial
on $[G,G]$, this implies that $\rho$ is trivial on $[G,G]\cap K^{r-1}$.
\begin{lem}
\label{lem:Commutator kernel containment}Assume that
either $G_{}=\GL_{2}(\cO_{r})$ or $G_{}=\SL_{2}(\cO_{r})$ and $p\neq2$.
Then 
\[
[G_{},G]\cap K^{r-1}\supseteq\SL_{2}(\cO_{r})\cap K^{r-1}=:\KSL^{r-1}.
\]
\end{lem}

\begin{proof}
When $r=2l$, direct calculation shows that for $1+\pi^{l}X\in K^{l}$
and $1+\pi^{l-1}Y\in K^{l-1}$, 
\[
[1+\pi^{l}X,1+\pi^{l-1}Y]=1+\pi^{r-1}(XY-YX).
\]
 Similarly, when $r=2l'+1$, for $1+\pi^{l'}X,1+\pi^{l'}Y\in K^{l'}$,
\[
[1+\pi^{l'}X,1+\pi^{l'}Y]=1+\pi^{r-1}(XY-YX).
\]
We have $\KSL^{r-1}=1+\pi^{r-1}\sl_{2}(\cO_{r})$ and it is well-known
that when $p\neq2$, $\sl_{2}(\F_{q})$ is a perfect Lie algebra.
Thus, since $1+\pi^{r-1}(XY-YX)$ only depends on the images of $X$
and $Y$ modulo $\mfp$, every element of $\KSL^{r-1}$ is in $[G,G]$,
hence in $[G,G]\cap K^{r-1}$.
\end{proof}
\begin{rem}
The above lemma holds even for $\GL_{2}(\cO_{r})$ when $\F_{q}=\F_{2}$,
as well as for $\SL_{2}(\cO_{r})$ when $\F_{q}=\F_{3}$, despite
the fact that in these cases we don't have $[G,G]\supseteq\SL_{2}(\cO_{r})$. 
\end{rem}

\begin{lem}
\label{lem:Factors through trivial on centre}Assume
that either $G_{}=\GL_{2}(\cO_{r})$ or $G_{}=\SL_{2}(\cO_{r})$ and
$p\neq2$. Let $\rho$ be an irreducible character of $G_{}$ that
factors through a character $\bar{\rho}$ of $G_{r-1}$. If $\rho$
is trivial on $Z_{}$, then $\bar{\rho}$ is trivial on $Z_{r-1}$. 
\end{lem}

\begin{proof}
A sufficient condition for the assertion is that $Z_{}K^{r-1}/K^{r-1}\cong Z_{r-1}$,
that is, that $\rho_{r-1}:Z_{}\rightarrow Z_{r-1}$ is surjective.
If $G=\GL_{2}(\cO_{r})$, $Z_{}\cong\cO_{r}^{\times}$, and it is
well-known that this maps surjectively onto $\cO_{r-1}^{\times}$.
For $G=\SL_{2}(\cO_{r})$ with $p$ odd, Hensel's lemma (or an easy
direct calculation) implies that any $a\in\cO_{i}$ such that $a^{2}=1$
lifts to an $\hat{a}\in\cO_{i+1}$ such that $a^{2}=1$, so also in
this case, $\rho_{r-1}$ maps $Z$ onto $Z_{r-1}$.
\end{proof}
The following lemma reduces our considerations, when $p\neq2$, to
the case where $\rho$ is twist-primitive. By Lemma~\ref{lem:constr_regular_chars_for_GL2_or_SL2}\,(\ref{enu: constr lemma i)})
this means that we will reduce to the case where $\rho$ is regular.
\begin{lem}
\label{lem:reducing to twist-primitive when p odd}Assume
that $r\geq2$, $p\neq2$ and that every irreducible character of
$G_{r-1}$ that is trivial on $Z_{r-1}$ is contained in the conjugation
character of $G_{r-1}$. Let $\rho\in\Irr(G)$ be a character that
is trivial on $Z$ and not twist-primitive. Then $\rho$ is contained
in the conjugation character of $G$. 
\end{lem}

\begin{proof}
As $\rho$ is not twist-primitive, it is trivial on $[G,G]\cap K^{r-1}$
(as shown in the beginning of this section). By Lemma~\ref{lem:Commutator kernel containment},
$[G,G]\cap K^{r-1}\supseteq\KSL^{r-1}$ (since $p\neq2$). Thus, when
$G=\SL_{2}(\cO_{r})$, $\rho$ factors through $G_{r-1}\cong\SL_{2}(\cO_{r})/\KSL^{r-1}$.
Next, asssume that $G=\GL_{2}(\cO_{r})$. As $\rho$ is also assumed
to be trivial on $Z$, it is trivial on $Z_{}\KSL^{r-1}$. But $Z_{}\geq\begin{pmatrix}1+x\pi^{r-1} & 0\\
0 & 1+x\pi^{r-1}
\end{pmatrix}$, so $\rho$ is trivial on the group of matrices of the form
\begin{align*}
 & \left\{ \begin{pmatrix}1+x\pi^{r-1} & 0\\
0 & 1+x\pi^{r-1}
\end{pmatrix}\begin{pmatrix}1+a\pi^{r-1} & b\pi^{r-1}\\
c\pi^{r-1} & 1-a\pi^{r-1}
\end{pmatrix}\right\} \\
 & =\left\{ \begin{pmatrix}1+(x+a)\pi^{r-1} & b\pi^{r-1}\\
c\pi^{r-1} & 1+(x-a)\pi^{r-1}
\end{pmatrix}\right\} .
\end{align*}
For any $y,z\in\cO_{r}$, there exists a solution $x,a\in\cO_{r}$
to the following system of congruences
\begin{align*}
x+a & =y\mod{\mfp},\\
x-a & =z\mod{\mfp},
\end{align*}
namely, $x=(y+z)2^{-1}$, $a=(y-x)2^{-1}$ (using that $p\neq2$,
so $2$ is invertible). Thus $Z\KSL^{r-1}\supseteq\KGL^{r-1}$ (in
fact, we have equality) and $\rho$ factors through $G_{r-1}\cong\GL_{2}(\cO_{r})/\KGL^{r-1}$.

In either case $\rho$ factors through $G_{r-1}$ and since $\rho$
is trivial on $Z$, Lemma~\ref{lem:Factors through trivial on centre}
implies that the character $\bar{\rho}$ that it defines on $G_{r-1}$
is trivial on $Z_{r-1}$. By hypothesis, $\bar{\rho}$ appears in
the conjugation character of $G_{r-1}$ and by \cite[Lemma~4]{tiep2023conjugation}
we conclude that $\rho$ appears in the conjugation character of $G$. 
\end{proof}
\begin{rem}
The above lemma, together with Lemma~\ref{lem:intersection CJ cap C(g) is ZU^1K^l cap C(g)},
are the only places where we need the assumption that $p\neq2$ for
$G=\GL_{2}(\cO_{r})$. The last paragraph of the proof of the above
lemma shows that for $G_{}=\GL_{2}(\cO_{r})$, for any $p$, all the
non-twist primitive characters that are also non-primitive and trivial
on $Z$ are contained in the conjugation character of $G$, assuming
that every irreducible character of $G_{r-1}$ that is trivial on
$Z_{r-1}$ is contained in the conjugation character of $G_{r-1}$.
For $G=\GL_{2}(\cO_{r})$ and $p=2$, the only remaining case is therefore
that of non-twist primitive characters that are primitive (i.e., those
whose orbit mod $\mfp$ contains a non-zero scalar). It seems likely
that the lemma is also true for these characters, but we have not
been able to prove it.
\end{rem}

\subsection{The case when $r$ is even}
\begin{thm}
\label{thm:p odd - r even}Let $r=2l$ and
assume that either $G_{}=\GL_{2}(\cO_{r})$ or $G_{}=\SL_{2}(\cO_{r})$
and $p\neq2$. For any regular $\rho\in\Irr(G_{})$ that is trivial
on $Z$ there exists an element $g_{\hat{\beta}}$ as in Section~\ref{sec:A-family-of} such that $\rho$ restricted to $C_G(g_{\hat{\beta}})$ contains the trivial character. 
\end{thm}

\begin{proof}
In this case we have $l=l'$. By Lemma~\ref{lem: tau can be taken as 0},
$\rho\in\Irr(G\mid\psi_{\beta})$ for some $\beta=\begin{pmatrix}0 & \lambda\\
\lambda^{-1}\Delta & 0
\end{pmatrix}\in\mfg(\cO_{l}),\lambda\in\cO_{l}^{\times}$. By Lemma~\ref{lem:constr_regular_chars_for_GL2_or_SL2},
$\rho=\Ind_{C_{G}(\hat{\beta})K^{l}}^{G_{}}\sigma$ for some $\sigma\in\Irr(C_{G}(\hat{\beta})K^{l}\mid\psi_{\beta})$
and any choice of lift $\hat{\beta}$. By Lemma~\ref{lem: is contained in the conj rep iff sigma }\,(\ref{enu: i)})
with $H=G_{}$ and $N=C_{G}(\hat{\beta})K^{l}$, it is enough to prove
that $\sigma$ restricted to $C_{G}(\hat{\beta})K^{l}\cap C_{G}(g)$
contains the trivial character, for some $g\in G$. We will prove
that $g$ can be taken to be $g_{\hat{\beta}}$ for some lift $\hat{\beta}$
(that depends on $\sigma$).

For any lift $\hat{\beta}=\begin{pmatrix}0 & \hat{\lambda}\\
\hat{\lambda}^{-1}\hat{\Delta} & 0
\end{pmatrix}$ of $\beta$, Lemma~\ref{lem:intersection of centralisers in ZUK}
says that
\[
C_{G}(\hat{\beta})K^{l}\cap C_{G_{}}(g_{\hat{\beta}})=ZU^{l-v}K^{l}\cap C_{G}(g_{\hat{\beta}}).
\]
It is thus sufficient to show the following.
\begin{claim*}
There exists a $\hat{\beta}=\begin{pmatrix}0 & \hat{\lambda}\\
\hat{\lambda}^{-1}\hat{\Delta} & 0
\end{pmatrix}$ such that $\sigma$ restricted to $ZU^{l-v}K^{l}\cap C_{G}(g_{\hat{\beta}})$
contains the trivial character.
\end{claim*}
Since $\sigma$ lies above $\psi_{\beta}$ and $U^{l-v}K^{l}\subseteq C_{G}(\hat{\beta})K^{l}$
by (\ref{eq:U^l'-v is in CK^l'}), Proposition~\ref{prop:There exist extns to U^l'-vK^l by formula and every extn is of this form}
implies that $\sigma|_{U^{l-v}K^{l}}$ contains an irreducible constituent
of the form $\psi_{\hat{\beta}}$ for some lift $\hat{\beta}=\begin{pmatrix}0 & \hat{\lambda}\\
\hat{\lambda}^{-1}\hat{\Delta} & 0
\end{pmatrix}$ of $\beta$ (using that $\tau=0$ here and that $\psi_{\hat{\beta}}$
is independent of the choice of lift of $\tau$). As $\rho$ is trivial
on $Z$, so is $\sigma$ (by Lemma~\ref{lem: rho trivial on Z iff sigma trivial on Z})
and by Lemma~\ref{lem: psi_hat-beta res to U^l'-vK^l cap C is the trivial char},
$\psi_{\hat{\beta}}$ is trivial on $U^{l-v}K^{l}\cap C_{G}(g_{\hat{\beta}})$.
This proves the claim and finishes the proof.
\end{proof}

\subsection{\protect\label{subsec:The-case-when-r-odd-v-less-l'}The case when
$r$ is odd and $v<l'$}

Throughout this section we assume that $r$ is odd. Let $\beta=\begin{pmatrix}0 & \lambda\\
\lambda^{-1}\Delta & 0
\end{pmatrix}\in\mfg(\cO_{l'}),\lambda\in\cO_{l'}^{\times}$ and let $\hat{\beta}=\begin{pmatrix}0 & \hat{\lambda}\\
\hat{\lambda}^{-1}\hat{\Delta} & 0
\end{pmatrix}\in\mfg(\cO_{r})$ be a lift of $\beta$. For $X\in\M_{2}(\cO_{i})$, $i\geq2$, we
will denote the image of $X$ in $\M_{2}(\F_{q})$ by $\bar{X}$.
It is well-known that the homomorphism $\eta:K^{l'}\rightarrow\M_{2}(\F_{q})$
given by $1+X\pi^{l'}\mapsto\bar{X}$, where $X\in\M_{2}(\cO_{r})$,
induces an isomorphism 
\begin{equation}
K^{l'}/K^{l}\cong\mfg(\F_{q}),\label{eq:K^l'/K^l =00003D g}
\end{equation}
where $\mfg(\F_{q})$ is viewed as a group under addition. As $K^{l'}/K^{l}$
is an elementary abelian $p$-group, this isomorphism can be promoted
to an $\F_{p}$-vector space isomorphism. This isomorphism is straightforward
when $G=\GL_{2}(\cO_{r})$. For $G=\SL_{2}(\cO_{r})$, if $1+\pi^{l'}X\in K^{l'}$
with $X=\begin{pmatrix}a & b\\
c & d
\end{pmatrix}\in\M_{2}(\cO_{r})$, then $1=\det(1+\pi^{l'}X)=1+(a+d)\pi^{l'}+(ad-bc)\pi^{r-1}$, so
$a+d\in\mfp^{l'}$, hence $a+d\equiv0\mod{\mfp}$, so $\bar{X}\in\sl_{2}(\F_{q})$.
We claim that the map $\eta$ is surjective. Indeed, for $\bar{X}=\begin{pmatrix}\bar{a} & \bar{b}\\
\bar{c} & \bar{d}
\end{pmatrix}\in\sl_{2}(\F_{q})$ with $\bar{a}\neq0$ or $\bar{d}\neq0$, Hensel's lemma implies that
there are $a,b,c,d\in\cO_{r}$ such that $a+d\equiv-(ad-bc)\pi^{l'}\mod{\pi^{l}}$
and thus $(a+d)\pi^{l'}\equiv-(ad-bc)\pi^{r-1}\mod{\pi^{r}}$, so
that for $x:=1+\pi^{l'}\begin{pmatrix}a & b\\
c & d
\end{pmatrix}$, we have $\det(x)=1$ and $\eta(x)=\bar{X}$. Moreover, if $\bar{a}=\bar{d}=0$,
then we may take $x:=1+\pi^{l'}\begin{pmatrix}\pi^{l'}bc & b\\
c & 0
\end{pmatrix}$. The kernel of $\eta$ is clearly $K^{l}$, so we have an isomorphism
$K^{l'}/K^{l}\iso\sl_{2}(\F_{q})$. 

For simplicity, and to conform with the notation in \cite{Stasinski-Stevens},
we will write 
\[
C^{i}=C_{G}(\hat{\beta})\cap K^{i},
\]
for $r\geq i\geq1$ (we will actually only use this notation for $i=1$
or $i=l'$). Since $C_{G}(\hat{\beta})$ is abelian and $[C^{1},K^{l'}]\subseteq K^{l}$,
the group $C^{1}K^{l}$ is normal in $C^{1}K^{l'}$. The quotient
\[
V:=\frac{C^{1}K^{l'}}{C^{1}K^{l}}=\frac{K^{l'}(C^{1}K^{l})}{C^{1}K^{l}}\cong\frac{K^{l'}}{K^{l'}\cap C^{1}K^{l}}=\frac{K^{l'}}{C^{l'}K^{l}}
\]
 is a quotient of $\mfg(\F_{q})$, so $V$ is itself an elementary
abelian $p$-group.

Let $\theta$ be an extension of $\psi_{\beta}$ to $C^{1}K^{l}$
(which exists, as $C$ is abelian and stabilises $\psi_{\beta}$).
Define the alternating bilinear form 
\[
h_{\beta}:V\times V\longrightarrow\C^{\times},\qquad h_{\beta}(x(C^{1}K^{l}),y(C^{1}K^{l}))=\theta([x,y])=\psi_{\beta}([x,y]),
\]
where $[x,y]=xyx^{-1}y^{-1}$ and the last equality holds because
the commutator subgroup of $C^{1}K^{l}$ lies in $K^{l}$. 
\begin{lem}
\label{lem:Heisenberg_induction}Assume that either $G=\GL_{2}(\cO_{r})$
or $G=\SL_{2}(\cO_{r})$ and $p\neq2$. For any extension $\theta$
of $\psi_{\beta}$ to $C^{1}K^{l}$, there exists a unique $\eta_{\theta}\in\Irr(C^{1}K^{l'}\mid\theta)$,
that is, $\Ind_{C^{1}K^{l}}^{C^{1}K^{l'}}\theta$ is a multiple of
$\eta_{\theta}$. Every element of $\Irr(C^{1}K^{l'}\mid\psi_{\beta})$
is of the form $\eta_{\theta}$.

Moreover, all of these assertions remain true when $C^{1}K^{l'}$
is replaced by $ZC^{1}K^{l'}$ and $C^{1}K^{l}$ is replaced by $ZC^{1}K^{l}$. 
\end{lem}

\begin{proof}
When $G=\GL_{2}(\cO_{r})$ we can deduce this from the more general
results in \cite{Stasinski-Stevens} (see also \cite[Section~4.2]{Alex_smooth_reps_GL2}
for the present special case). Indeed, \cite[p.~1078]{Stasinski-Stevens}
shows that $\theta$ is fixed by $C^{1}K^{l'}$ and \cite[Lemmas~4.3, 4.4 and 4.5 ii)]{Stasinski-Stevens}
show that the form $h_{\beta}$ is non-degenerate. The existence and
uniqueness of $\eta_{\theta}$ then follows from the general result
\cite[Lemma~3.2]{Stasinski-Stevens}. To help translating these results
from \cite{Stasinski-Stevens} into our present situation, take $\mfA=\mfA_{\mathrm{M}}=\M_{2}(\cO_{r})$,
$\mfP=\mfp\mfA$ so that $e=1$ and $U^{i}$ is $K^{i}$ in our present
notation. Then the map $\rho:K^{l'}\rightarrow K^{l'}/K^{l}\iso\mfA/\mfP$
in \cite[Lemma~4.4]{Stasinski-Stevens} is our map $\eta$ above and
$\rho^{-1}(C_{\mfA/\mfP}(\beta+\mfP))$ equals $C^{l'}K^{l}$.

Suppose now that $G=\SL_{2}(\cO_{r})$ and $p\neq2$. The proofs of
the above results in \cite{Stasinski-Stevens} go through\emph{ mutatis
mutandis} for $\mfA=\sl_{2}(\cO_{r})$ and all corresponding groups
changed to their respective versions in $\SL_{2}(\cO_{r})$. The condition
$p\neq2$ is needed for the non-degeneracy of the trace form in \cite[Lemma~2.5]{Stasinski-Stevens}
(when $E=\sl_{2}(\cO_{r})$), which is used in \cite[Lemma~3.1 i)]{Stasinski-Stevens},
which in turn is used in the proof of \cite[Lemma~4.4]{Stasinski-Stevens}.

Now, any $\sigma\in\Irr(C^{1}K^{l'}\mid\psi_{\beta})$ is of the form
$\eta_{\theta}$ because $\sigma|_{C^{1}K^{l'}}$ contains some $\theta\in\Irr(C^{1}K^{l}\mid\psi_{\beta})$
so by the uniqueness of $\eta_{\theta}$ we must have $\sigma=\eta_{\theta}$. 

Moreover, 
\[
\frac{ZC^{1}K^{l'}}{ZC^{1}K^{l}}=\frac{K^{l'}(ZC^{1}K^{l})}{ZC^{1}K^{l}}\cong\frac{K^{l'}}{K^{l'}\cap ZC^{1}K^{l}}=\frac{K^{l'}}{C^{l'}K^{l}}\cong V,
\]
the corresponding bilinear form on $\frac{ZC^{1}K^{l'}}{ZC^{1}K^{l}}$
has the same values as $h_{\beta}$ on $V$ and any extension of $\psi_{\beta}$
to $ZC^{1}K^{l}$ is fixed by $ZC^{1}K^{l'}$ if and only if its restriction
to $C^{1}K^{l}$ is fixed by $C^{1}K^{l'}$. Thus the last assertion
follows from the above.
\end{proof}
\begin{thm}
\label{thm:p odd - r odd - v < l'}Assume
that $r$ is odd and $v<l'$. Assume moreover that either $G=\GL_{2}(\cO_{r})$
or $G=\SL_{2}(\cO_{r})$ and $p\neq2$. For any regular $\rho\in\Irr(G_{})$ that is trivial
on $Z$ there exists an element $g_{\hat{\beta}}$ as in Section~\ref{sec:A-family-of} such that $\rho$ restricted to $C_G(g_{\hat{\beta}})$ contains the trivial character.
\end{thm}

\begin{proof}
By Lemma~\ref{lem: tau can be taken as 0},
$\rho\in\Irr(G\mid\psi_{\beta})$ for some $\beta=\begin{pmatrix}0 & \lambda\\
\lambda^{-1}\Delta & 0
\end{pmatrix}\in\mfg(\cO_{l}),\lambda\in\cO_{l}^{\times}$. By Lemma~\ref{lem:constr_regular_chars_for_GL2_or_SL2},
$\rho=\Ind_{C_{G}(\hat{\beta})K^{l'}}^{G_{}}\sigma$ for some $\sigma\in\Irr(C_{G}(\hat{\beta})K^{l'}\mid\psi_{\beta})$
and any choice of lift $\hat{\beta}$. By Lemma~\ref{lem: is contained in the conj rep iff sigma }\,(\ref{enu: i)})
with $H=G$ and $N=C_{G}(\hat{\beta})K^{l'}$, it is enough to prove
that $\sigma$ restricted to $C_{G}(\hat{\beta})K^{l'}\cap C_{G}(g)$
contains the trivial character, for some $g\in G$. We will prove
that $g$ can be taken to be $g_{\hat{\beta}}$ for some lift $\hat{\beta}$.

For any lift $\hat{\beta}=\begin{pmatrix}0 & \hat{\lambda}\\
\hat{\lambda}^{-1}\hat{\Delta} & 0
\end{pmatrix}$ of $\beta$, Lemma~\ref{lem:intersection of centralisers in ZUK}
says that
\[
C_{G}(\hat{\beta})K^{l'}\cap C_{G}(g_{\hat{\beta}})=ZU^{l'-v}K^{l'}\cap C_{G}(g_{\hat{\beta}}).
\]
Moreover, since $v<l'$ we have $ZU^{l'-v}K^{l'}\subseteq ZK^{1}$,
so that, using, \ref{eq:U^l'-v is in CK^l'} $ZU^{l'-v}K^{l'}\subseteq C_{G}(\hat{\beta})K^{l'}\cap ZK^{1}\subseteq ZC^{1}K^{l'}$.
Thus
\[
ZC^{1}K^{l'}\cap C_{G}(g_{\hat{\beta}})\subseteq C_{G}(\hat{\beta})K^{l'}\cap C_{G}(g_{\hat{\beta}})=ZU^{l'-v}K^{l'}\cap C_{G}(g_{\hat{\beta}})\subseteq ZC^{1}K^{l'}\cap C_{G}(g_{\hat{\beta}}),
\]
so that equalities holds everywhere. It is thus sufficient to show
the following.
\begin{claim*}
There exists a lift $\hat{\beta}=\begin{pmatrix}0 & \hat{\lambda}\\
\hat{\lambda}^{-1}\hat{\Delta} & 0
\end{pmatrix}$ such that $\sigma$ contains the trivial character when restricted
to $ZC^{1}K^{l'}\cap C_{G}(g_{\hat{\beta}})$.
\end{claim*}
By Lemma~\ref{lem:Heisenberg_induction}, there exists an irreducible
constituent of $\sigma|_{ZC^{1}K^{l'}}$ of the form $\eta_{\theta}$, such that $\Ind_{ZC^{1}K^{l}}^{ZC^{1}K^{l'}}\theta$
is a multiple of $\eta_{\theta}$,
for some extension $\theta\in\Irr(ZC^{1}K^{l}\mid\psi_{\beta})$ of
$\psi_{\beta}$. By Proposition~\ref{prop:There exist extns to U^l'-vK^l by formula and every extn is of this form},
\[
\theta|_{U^{l'-v}K^{l}}=\psi_{\hat{\beta}},
\]
for some $\hat{\beta}=\begin{pmatrix}0 & \hat{\lambda}\\
\hat{\lambda}^{-1}\hat{\Delta} & 0
\end{pmatrix}$, which we fix throughout the rest of the proof. By Lemma~\ref{lem: is contained in the conj rep iff sigma }\,(\ref{enu: ii)})
with $H=ZC^{1}K^{l'}$, $N=ZC^{1}K^{l}$ and $C=C_{G}(g_{\hat{\beta}})$,
the above claim follows from the following claim.
\begin{claim*}
The restriction of $\theta$ to $ZC^{1}K^{l}\cap C_{G}(g_{\hat{\beta}})$
is trivial.
\end{claim*}
Now, since $v<l'$ we have $l-w=l-v\geq1$, so Lemma~\ref{lem:intersection of centralisers in ZUK}
implies that 
\[
ZC^{1}K^{l}\cap C_{G}(g_{\hat{\beta}})=ZU^{l-v}K^{l}\cap C_{G}(g_{\hat{\beta}}).
\]
As $\rho$ is trivial on $Z$, so is $\sigma$ (by Lemma~\ref{lem: rho trivial on Z iff sigma trivial on Z}),
hence $\eta_{\theta}$, and therefore $\theta$, are trivial on $Z$.
The second claim therefore follows from the fact that $\psi_{\hat{\beta}}$
restricted to
\[
U^{l-v}K^{l}\cap C_{G}(g_{\hat{\beta}})
\]
is trivial (Lemma~\ref{lem: psi_hat-beta res to U^l'-vK^l cap C is the trivial char}).
This implies the first claim and finishes the proof.
\end{proof}
\begin{rem}
The condition $v<l'$ is needed in the above proof, since the inclusion
$ZU^{l'-v}K^{l'}\cap C_{G}(g_{\hat{\beta}})\subseteq ZC^{1}K^{l'}\cap C_{G}(g_{\hat{\beta}})$
fails for $v=l'$. An alternative approach would be to use Heisenberg
lifts from $U^{l'-v}K^{l}$ to $U^{l'-v}K^{l'}$, but this again fails
for $v=l'$, as $UK^{l}$ is not normal in $UK^{l'}$. This is why
we need an entirely different approach for the case $v=l'$.
\end{rem}

\subsection{\protect\label{subsec:The-case-when-r-odd-v-bigger-0}The case when
$r$ is odd and $v>0$}

In this section we continue to assume that $r$ is odd. We will deal
with the case when $\Delta\in\mfp$ and $p\neq2$. Even though we
only need the case $v=l'$, everything in this section is proved for
all $v>0$ without any additional difficulty.

Let $\beta=\begin{pmatrix}0 & \lambda\\
\lambda^{-1}\Delta & 0
\end{pmatrix}\in\mfg(\cO_{l'}),\lambda\in\cO_{l'}^{\times}$ and let $\hat{\beta}=\begin{pmatrix}0 & \hat{\lambda}\\
\hat{\lambda}^{-1}\hat{\Delta} & 0
\end{pmatrix}\in\mfg(\cO_{r})$ be a lift of $\beta$. We have an alternating bilinear form on $K^{l'}/K^{l}\cong\mfg(\F_{q})$
given by
\[
h_{\beta}:K^{l'}/K^{l}\times K^{l'}/K^{l}\longrightarrow\C,\qquad f(xK^{l},yK^{l})=\psi_{\beta}([x,y]).
\]
Let
\[
J:=B^{l'}K^{l}
\]
 where $B^{l'}=\begin{pmatrix}1+\mfp^{l'} & \mfp^{l'}\\
0 & 1+\mfp^{l'}
\end{pmatrix}\cap G$. The following lemma is well known.
\begin{lem}
\label{lem:Heisenberg lemma for K^l'/K^l}Assume
that $G=\GL_{2}$ or $G=\SL_{2}$ and $p\neq2$. Assume moreover that
$\Delta\in\mfp$. Then the radical of $h_{\beta}$ is $Z^{l'}U^{l'}K^{l}/K^{l}$
and $J/K^{l}$ is a maximal isotropic subspace for $h_{\beta}$. Thus,
for every $\eta\in\Irr(K^{l'}\mid\psi_{\beta})$, there exists an
extension $\tilde{\psi}_{\beta}\in\Irr(J\mid\psi_{\beta})$ of $\psi_{\beta}$
such such that $\eta=\Ind_{J}^{K^{l'}}\tilde{\psi}_{\beta}$. 
\end{lem}

\begin{proof}
Note that 
\[
(C_{G}(\hat{\beta})\cap K^{l'})K^{l}=\rho^{-1}(C_{\M_{2}(\F_{q})}(\bar{\beta}))=Z^{l'}U^{l'}K^{l}.
\]
Thus the assertion about the radical follows from \cite[Proposition~4.2(1)]{Hill_regular}
for $\GL_{2}$ and we observe that the same argument goes through
for $\SL_{2}$ when $p\neq2$, as the trace form is non-degenerate
in this case.

The assertion about $J/K^{l}$ follows from the proof of \cite[Lemma~4.5]{Hill_regular},
or by a direct computation: It is an isotropic subspace because for
any $x,y\in J$, $[x,y]\in U^{r-1}\subseteq U^{l'}$ and it is maximal
because 
\[
J/K^{l}\cong\begin{cases}
\left(\begin{smallmatrix}* & *\\
0 & *
\end{smallmatrix}\right) & \text{when }G=\GL_{2},\\
\left(\begin{smallmatrix}a & *\\
0 & -a
\end{smallmatrix}\right),\ a\in\F_{q} & \text{when }G=\SL_{2}
\end{cases}
\]
and thus in either case $J/K^{l}$ has codimension $1$ in $\mfg(\F_{q})$.

The last assertion follows from \cite[Proposition~4.2]{Hill_regular}
(the same proof works for $G=\SL_{2}$).
\end{proof}
\begin{lem}
\label{lem:every char of U^l'K^l above psi_beta is =00005Cpsi_hat-beta}The
function $\psi_{\hat{\beta}}:U^{l'}K^{l}\rightarrow\C^{\times}$ given
by
\[
\psi_{\hat{\beta}}(x)=\psi(\Tr(\hat{\beta}(x-1))),\qquad\text{for }x\in U^{l'}K^{l},
\]
is a linear character. If $\hat{\beta}'=\begin{pmatrix}0 & \hat{\lambda}'\\
(\hat{\lambda}')^{-1}\hat{\Delta}' & 0
\end{pmatrix}\in\M_{2}(\cO_{r})$ is another lift of $\beta$, then $\psi_{\hat{\beta}}=\psi_{\hat{\beta}'}$
if and only if $\hat{\lambda}^{-1}\hat{\Delta}\equiv(\hat{\lambda}')^{-1}\hat{\Delta}'\mod{\mfp}^{l}$.
Moreover, every element of $\Irr(U^{l'}K^{l}\mid\psi_{\beta})$ is
of the form $\psi_{\hat{\beta}}$.
\end{lem}

\begin{proof}
Any element of $U^{l'}K^{l}$ has the form 
\[
\begin{pmatrix}1+a\pi^{l} & b\pi^{l'}\\
c\pi^{l} & 1+d\pi^{l}
\end{pmatrix}
\]
for some $a,b,c,d\in\cO_{r}$ and 
\[
\psi_{\hat{\beta}}\left(\begin{pmatrix}1+a\pi^{l} & b\pi^{l'}\\
c\pi^{l} & 1+d\pi^{l}
\end{pmatrix}\right)=\psi\left(\Tr\left(\hat{\beta}\begin{pmatrix}a\pi^{l} & b\pi^{l'}\\
c\pi^{l} & d\pi^{l}
\end{pmatrix}\right)\right)=\psi(\hat{\lambda}c\pi^{l}+b\hat{\lambda}^{-1}\hat{\Delta}\pi^{l'}).
\]
Note that $\hat{\lambda}c\pi^{l}$ is independent of the choice of
lift $\hat{\lambda}$ of $\lambda$ and therefore $\psi_{\hat{\beta}}$,
for fixed $\beta$, only depends on the image of $\hat{\lambda}^{-1}\hat{\Delta}$
mod $\mfp^{l}$. 

We first show that $\psi_{\hat{\beta}}$ is a linear character of
$U^{l'}K^{l}$. For any $a',b',c',d'\in\cO_{r}$,
\begin{align*}
 & \psi_{\hat{\beta}}\left(\begin{pmatrix}1+a\pi^{l} & b\pi^{l'}\\
c\pi^{l} & 1+d\pi^{l}
\end{pmatrix}\begin{pmatrix}1+a'\pi^{l} & b'\pi^{l'}\\
c'\pi^{l} & 1+d'\pi^{l}
\end{pmatrix}\right)\\
 & =\psi_{\hat{\beta}}\left(\begin{pmatrix}\cdots & (b+b')\pi^{l'}\\
(c+c')\pi^{l} & \cdots
\end{pmatrix}\right)\\
 & =\psi(\hat{\lambda}(c+c')\pi^{l}+\hat{\lambda}^{-1}\hat{\Delta}(b+b')\pi^{l'})\\
 & =\psi(\hat{\lambda}c\pi^{l}+\hat{\lambda}^{-1}\hat{\Delta}b\pi^{l'})\psi(\hat{\lambda}c'\pi^{l}+\hat{\lambda}^{-1}\hat{\Delta}b'\pi^{l'}).
\end{align*}
Thus $\psi_{\hat{\beta}}$ is a homomorphism.

We have already seen that $\psi_{\hat{\beta}}$ only depends on $\hat{\lambda}^{-1}\hat{\Delta}$
mod $\mfp^{l}$. Conversely, if $\psi_{\hat{\beta}}=\psi_{\hat{\beta}'}$,
then
\[
\psi(\hat{\lambda}c\pi^{l}+\hat{\lambda}^{-1}\hat{\Delta}b\pi^{l'})=\psi(\hat{\lambda}'c\pi^{l}+(\hat{\lambda}')^{-1}\hat{\Delta}'b\pi^{l'}),
\]
for all $b,c\in\cO_{r}$ ($\hat{\lambda}c\pi^{l}=\hat{\lambda}'c\pi^{l}$
in $\cO_{r}$), so $\psi(b\pi^{l'}(\hat{\lambda}^{-1}\hat{\Delta}-(\hat{\lambda}')^{-1}\hat{\Delta}'))=1$
for all $b$; hence $\pi^{l'}(\hat{\lambda}^{-1}\hat{\Delta}-(\hat{\lambda}')^{-1}\hat{\Delta}')=0$
(as $\psi$ is non-trivial on $\mfp^{r-1}$), that is, 
\[
\hat{\lambda}^{-1}\hat{\Delta}-(\hat{\lambda}')^{-1}\hat{\Delta}'\in\mfp^{l}.
\]
We now show that every element of $\Irr(U^{l'}K^{l}\mid\psi_{\beta})$
is of the form $\psi_{\hat{\beta}}$. The number of elements of $\Irr(U^{l'}K^{l}\mid\psi_{\beta})$
is $|U^{l'}K^{l}/K^{l}|=|U^{l'}/U^{l}|=q$. But we have shown that
the number of distinct characters $\psi_{\hat{\beta}}$ is equal to
the number of distinct lifts of $\lambda^{-1}\Delta\in\cO_{l'}$ to
$\cO_{l}$, which also equals $q$. Thus the $\psi_{\hat{\beta}}$
exhaust $\Irr(U^{l'}K^{l}\mid\psi_{\beta})$.
\end{proof}
\begin{lem}
\label{lem:psi_beta on the radical extends to tilde-psi_beta on J}Assume
that $\Delta\in\mfp$. The function 
\[
\tilde{\psi}_{\hat{\beta}}(x)=\psi(\Tr(\hat{\beta}(x-1))),\qquad\text{for }x\in J,
\]
 is a linear character of $J$ and thus every element of $\Irr(U^{l'}K^{l}\mid\psi_{\beta})$
has an extension to $J$ of the form $\tilde{\psi}_{\hat{\beta}}$.
Moreover, $\tilde{\psi}_{\hat{\beta}}$ is $C_{G}(\hat{\beta})$-stable.
\end{lem}

\begin{proof}
Any element of $J$ has the form
\[
\begin{pmatrix}1+a\pi^{l'} & b\pi^{l'}\\
c\pi^{l} & 1+d\pi^{l'}
\end{pmatrix}
\]
for some $a,b,c,d\in\cO_{r}$ and
\[
\tilde{\psi}_{\hat{\beta}}\left(\begin{pmatrix}1+a\pi^{l'} & b\pi^{l'}\\
c\pi^{l} & 1+d\pi^{l'}
\end{pmatrix}\right)=\psi\left(\Tr\left(\hat{\beta}\begin{pmatrix}a\pi^{l'} & b\pi^{l'}\\
c\pi^{l} & d\pi^{l'}
\end{pmatrix}\right)\right)=\psi(\hat{\lambda}c\pi^{l}+\hat{\lambda}^{-1}\hat{\Delta}b\pi^{l'}).
\]
Thus, for any $a',b',c',d'\in\cO_{r}$,
\begin{align*}
 & \tilde{\psi}_{\beta'}\left(\begin{pmatrix}1+a\pi^{l'} & b\pi^{l'}\\
c\pi^{l} & 1+d\pi^{l'}
\end{pmatrix}\begin{pmatrix}1+a'\pi^{l'} & b'\pi^{l'}\\
c'\pi^{l} & 1+d'\pi^{l'}
\end{pmatrix}\right)\\
 & =\tilde{\psi}_{\beta'}\left(\begin{pmatrix}\cdots & ((b+b')+(ab'+bd')\pi^{l'})\pi^{l'}\\
(c+c')\pi^{l} & \cdots
\end{pmatrix}\right)\\
 & =\psi(\hat{\lambda}(c+c')\pi^{l}+\hat{\lambda}^{-1}\hat{\Delta}((b+b')\pi^{l'}+(ab'+bd')\pi^{r-1}))\\
 & =\psi(\hat{\lambda}(c+c')\pi^{l}+\hat{\lambda}^{-1}\hat{\Delta}(b+b')\pi^{l'})\\
 & \quad(\text{since }\text{\ensuremath{\hat{\Delta}\in\mfp} and }\psi(\mfp^{r})=1)\\
 & =\psi(\hat{\lambda}c\pi^{l}+\hat{\lambda}^{-1}\hat{\Delta}b\pi^{l'})\psi(\hat{\lambda}c'\pi^{l}+\hat{\lambda}^{-1}\hat{\Delta}b'\pi^{l'}).
\end{align*}
Thus $\tilde{\psi}_{\hat{\beta}}$ is a homomorphism and it is clear
by definition that it is an extension of $\psi_{\hat{\beta}}\in\Irr(U^{l'}K^{l})$. 

Finally, $\tilde{\psi}_{\hat{\beta}}$ is $C_{G}(\hat{\beta})$-stable
because if $c\in C_{G}(\hat{\beta})$ and $x\in J$, then
\begin{align*}
\tilde{\psi}_{\hat{\beta}}(c^{-1}xc) & =\psi(\Tr(\hat{\beta}(c^{-1}xc-1)))=\psi(\Tr(c^{-1}\hat{\beta}(x-1)c)))\\
 & =\psi(\Tr(\hat{\beta}(x-1)))=\tilde{\psi}_{\hat{\beta}}(x).
\end{align*}

\end{proof}
\begin{lem}
\label{lem:J normal and tilde-psi_beta extends}Assume
that $\Delta\in\mfp$. The group $J$ is normalised by $C_{G}(\hat{\beta})$
and thus $\tilde{\psi}_{\hat{\beta}}$ extends to $C_{G}(\hat{\beta})J$.
\end{lem}

\begin{proof}
As $\Delta\in\mfp$, we have $C_{G}(\hat{\beta})\subseteq BK^{1}$,
where $B$ is the upper triangular subgroup of $G$. We show that
$BK^{1}$ normalises $J$. Indeed, $B$ normalises $J=B^{l'}K^{l}$,
as $B$ contains $B^{l'}$ and $K^{l}$ is normal in $G$; moreover,
$K^{1}$ normalises $J$ since $[K^{1},B^{l'}]\subseteq K^{l}$. Since
$\hat{\beta}$ is regular in $\M_{2}(\cO_{r})$, $C_{G}(\hat{\beta})$
is abelian and the conclusion follows.
\end{proof}
\begin{prop}
\label{prop:sigma induced from CJ}Assume that $G=\GL_{2}$
or $G=\SL_{2}$ and $p\neq2$. Assume moreover that $\Delta\in\mfp$.
Write $S$ for $\Stab_{G}(\psi_{\beta})$. For every $\sigma\in\Irr(S\mid\psi_{\beta})$
there exists a lift $\hat{\beta}=\begin{pmatrix}0 & \hat{\lambda}\\
\hat{\lambda}^{-1}\hat{\Delta} & 0
\end{pmatrix}\in\mfg(\cO_{r})$ and an extension $\omega$ of $\tilde{\psi}_{\hat{\beta}}$ to $CJ$,
where $C:=C_{G}(\hat{\beta})$, such that
\[
\sigma=\Ind_{CJ}^{S}\omega.
\]
\end{prop}

\begin{proof}
Let $\sigma\in\Irr(S\mid\psi_{\beta})$. Recall that $S=C_{G}(\hat{\beta})K^{l'}$
for \emph{any} lift $\hat{\beta}\in\mfg(\cO_{r})$. Then $\sigma|_{K^{l'}}$
has an irreducible constituent that lies above $\psi_{\beta}$, so
by Lemma~\ref{lem:Heisenberg lemma for K^l'/K^l},
together with Lemma~\ref{lem:every char of U^l'K^l above psi_beta is =00005Cpsi_hat-beta}
and Lemma~\ref{lem:psi_beta on the radical extends to tilde-psi_beta on J},
$\sigma|_{K^{l'}}$ has an irreducible constituent of the form $\Ind_{J}^{K^{l'}}\tilde{\psi}_{\hat{\beta}}$.
Thus, by Frobenius reciprocity, $\sigma$ is an irreducible constituent
of 
\[
\Ind_{J}^{S}\tilde{\psi}_{\hat{\beta}}=\Ind_{CJ}^{S}\Ind_{J}^{CJ}\tilde{\psi}_{\hat{\beta}}=\sum_{\omega\in\Irr(CJ\mid\tilde{\psi}_{\hat{\beta}})}\Ind_{CJ}^{S}\omega.
\]
By Lemma~\ref{lem:J normal and tilde-psi_beta extends},
$\Irr(CJ\mid\tilde{\psi}_{\hat{\beta}})$ consists of all the extensions
of $\tilde{\psi}_{\hat{\beta}}$ to $CJ$, so in particular, every
$\omega$ is of degree one. As $\sigma$ is irreducible, it is contained
in some $\Ind_{CJ}^{S}\omega$, so $\deg\sigma\leq[S:CJ]$. On the
other hand, since $\sigma|_{K^{l'}}$ contains $\Ind_{J}^{K^{l'}}\tilde{\psi}_{\hat{\beta}}$,
we must have $\deg\sigma\geq[K^{l'}:J]$. But since $\hat{\Delta}\in\mfp$,
\[
C\cap J=\left\{ \begin{pmatrix}a & b\hat{\lambda}\\
b\hat{\lambda}^{-1}\hat{\Delta} & a
\end{pmatrix}\mid a\in1+\mfp^{l'},b\in\mfp^{l'}\right\} \cap G=C\cap K^{l'},
\]
so
\[
[CK^{l'}:CJ]=\frac{|C|/|C\cap K^{l'}|\cdot|K^{l'}|}{|C|/|C\cap J|\cdot|J|}=\frac{|K^{l'}|}{|J|}=[K^{l'}:J]=[S:CJ].
\]
Thus $\deg\sigma=[S:CJ]$ and therefore
\[
\sigma=\Ind_{CJ}^{S}\omega.
\]
\end{proof}
Recall from just before Lemma~(\ref{lem:intersection of centralisers in ZUK})
that we use $w$ to denote $\val(\hat{\Delta}_{l})$ and that $w=v$
whenever $v<l'$.
\begin{lem}
\label{lem:intersection CJ cap C(g) is ZU^1K^l cap C(g)}Assume
that $\Delta\in\mfp$. If $w=l$ or $w<l$ and $p\neq2$, then $ZU^{l-w}K^{l}\subseteq C_{G}(\hat{\beta})J$
and 
\[
C_{G}(\hat{\beta})J\cap C_{G}(g_{\hat{\beta}})=ZU^{l-w}K^{l}\cap C_{G}(g_{\hat{\beta}}).
\]
Moreover, if $\hat{\beta}'=\begin{pmatrix}0 & \hat{\lambda}'\\
(\hat{\lambda}')^{-1}\hat{\Delta}' & 0
\end{pmatrix}\in\M_{2}(\cO_{r})$ is another lift of $\beta$ such that $\hat{\lambda}^{-1}\hat{\Delta}\equiv(\hat{\lambda}')^{-1}\hat{\Delta}'\mod{\mfp}^{l}$,
then 
\[
ZU^{l-w}K^{l}\cap C_{G}(g_{\hat{\beta}})=ZU^{l-w}K^{l}\cap C_{G}(g_{\hat{\beta}'}).
\]
\end{lem}

\begin{proof}
Let $h\in C_{G}(\hat{\beta})J\cap C_{G}(g_{\hat{\beta}})$. Then 
\begin{align*}
h_{l} & \in C_{G_{l}}(\hat{\beta}_{l})\Big(\begin{smallmatrix}1+\mfp^{l'} & \mfp^{l'}\\
0 & 1+\mfp^{l'}
\end{smallmatrix}\Big)\cap C_{G_{l}}\Big(\begin{smallmatrix}1 & \hat{\lambda}_{l}\\
-\hat{\lambda}_{l}^{-1}\hat{\Delta}_{l} & 1-\hat{\Delta}_{l}
\end{smallmatrix}\Big)\\
 & \subseteq\left\{ \begin{pmatrix}a+\mfp^{l'} & b\hat{\lambda}_{l}+\mfp^{l'}\\
b\hat{\lambda}_{l}^{-1}\hat{\Delta}_{l} & a+\mfp^{l'}
\end{pmatrix}\mid a,b\in\cO_{l}\right\} \cap\left\{ \begin{pmatrix}x & y\hat{\lambda}_{l}\\
-y\hat{\lambda}_{l}^{-1}\hat{\Delta}_{l} & x-y\hat{\Delta}_{l}
\end{pmatrix}\mid x,y\in\cO_{l}\right\}.
\end{align*}
A matrix in this intersection is of the form
\[
\begin{pmatrix}a+s\pi^{l'} & b\hat{\lambda}_{l}+t\pi^{l'}\\
b\hat{\lambda}_{l}^{-1}\hat{\Delta}_{l} & a+u\pi^{l'}
\end{pmatrix}=\begin{pmatrix}x & y\hat{\lambda}_{l}\\
-y\hat{\lambda}_{l}^{-1}\hat{\Delta}_{l} & x-y\hat{\Delta}_{l}
\end{pmatrix}.
\]
 If $\hat{\Delta}_{l}=0$ (i.e., $w=l$), then 
\[
h_{l}\in\left\{ \begin{pmatrix}x & y\\
0 & x
\end{pmatrix}\mid x,y\in\cO_{l}\right\} \cap G_{l},
\]
so that $h\in ZUK^{l}=ZU^{0}K^{l}$. 

If $w<l$, then $b\hat{\lambda}_{l}^{-1}\hat{\Delta}_{l}=-y\hat{\lambda}_{l}^{-1}\hat{\Delta}_{l}$
implies that $b\equiv-y\mod{\mfp^{l-w}}$. We also have $b\hat{\lambda}_{l}+t\pi^{l'}=y\hat{\lambda}_{l}$,
which implies that $b\equiv y\mod{\mfp^{l'}}$. Since $l-w\leq l'$
($w\geq1$ since $\Delta\in\mfp$), we obtain $2y\equiv0\mod{\mfp}^{l-w}$
and as $p\neq2$, we deduce that
\[
y\in\mfp^{l-w}.
\]
This implies that $y\hat{\Delta}_{l}=0$ and thus 
\[
h_{l}\in\left\{ \begin{pmatrix}x & y\\
0 & x
\end{pmatrix}\mid x\in\cO_{l},y\in\mfp^{l-w}\right\} \cap G_{l},
\]
so that $h\in ZU^{l-w}K^{l}$. 

Thus, in either case, $C_{G}(\hat{\beta})J\cap C_{G}(g_{\hat{\beta}})\subseteq ZU^{l-w}K^{l}\cap C_{G}(g_{\hat{\beta}})$.
Furthermore, for any $h\in ZU^{l-w}K^{l}$,
\begin{align*}
h_{l}\in & Z_{l}U_{l}^{l-w}=\left\{ \begin{pmatrix}x & y\\
0 & x
\end{pmatrix}\mid x,y\in\cO_{l},y\in\mfp^{l-w}\right\} \cap G_{l}\subseteq C_{G_{l}}(\hat{\beta}_{l}),
\end{align*}
so that $h\in C_{G}(\hat{\beta})K^{l}\subseteq C_{G}(\hat{\beta})J$.
Hence $ZU^{l-w}K^{l}\subseteq C_{G}(\hat{\beta})J$ and therefore
$C_{G}(\hat{\beta})J\cap C_{G}(g_{\hat{\beta}})=ZU^{l-w}K^{l}\cap C_{G}(g_{\hat{\beta}})$.

Finally, for $\hat{\beta}'$ as in the statement, we have $\val(\hat{\Delta}_{l})=\val(\hat{\Delta}_{l}')$
(that is, the same $w$), so by what we have already proved, 
\[
ZU^{l-w}K^{l}\cap C_{G}(g_{\hat{\beta}})=C_{G}(\hat{\beta})J\cap C_{G}(g_{\hat{\beta}})=C_{G}(\hat{\beta}')J\cap C_{G}(g_{\hat{\beta}'})=ZU^{l-w}K^{l}\cap C_{G}(g_{\hat{\beta}'}).
\]
\end{proof}
\begin{thm}
\label{thm:p even - r odd - v=00003Dl'}Assume
that $r$ is odd and $\Delta\in\mfp$. Assume moreover that either
$G=\GL_{2}(\cO_{r})$ or $G=\SL_{2}(\cO_{r})$ and $p\neq2$. For any regular $\rho\in\Irr(G_{})$ that is trivial
on $Z$ there exists an element $g_{\hat{\beta}}$ as in Section~\ref{sec:A-family-of} such that $\rho$ restricted to $C_G(g_{\hat{\beta}})$ contains the trivial character.
\end{thm}

\begin{proof}
By Lemma~\ref{lem: tau can be taken as 0},
$\rho\in\Irr(G\mid\psi_{\beta})$ for some $\beta=\begin{pmatrix}0 & \lambda\\
\lambda^{-1}\Delta & 0
\end{pmatrix}\in\mfg(\cO_{l}),\lambda\in\cO_{l}^{\times}$. By Lemma~\ref{lem:constr_regular_chars_for_GL2_or_SL2}
and Proposition~\ref{prop:sigma induced from CJ},
$\rho=\Ind_{C_{G}(\hat{\beta})J}^{G}\omega$ for some lift $\hat{\beta}=\begin{pmatrix}0 & \hat{\lambda}\\
\hat{\lambda}^{-1}\hat{\Delta} & 0
\end{pmatrix}\in\sl_{2}(\cO_{r})$ and an extension $\omega$ of $\tilde{\psi}_{\hat{\beta}}$ to $C_{G}(\hat{\beta})J$. 

By Lemma~\ref{lem: is contained in the conj rep iff sigma }\,(\ref{enu: i)})
with $H=G$ and $N=C_{G}(\hat{\beta})J$, it is enough to prove that
$\omega$ restricted to $C_{G}(\hat{\beta})J\cap C_{G}(g)$ is the
trivial character, for some $g\in G$. By Lemma~\ref{lem:intersection CJ cap C(g) is ZU^1K^l cap C(g)},
\[
C_{G}(\hat{\beta})J\cap C_{G}(g_{\hat{\beta}})=ZU^{l-w}K^{l}\cap C_{G}(g_{\hat{\beta}}),
\]
as well as $U^{l-w}K^{l}\subseteq C_{G}(\hat{\beta})J$. It is thus
sufficient to show the following.
\begin{claim*}
$\omega$ restricted to $ZU^{l-w}K^{l}\cap C_{G}(g_{\hat{\beta}})$
is the trivial character.
\end{claim*}
By Proposition~\ref{prop:There exist extns to U^l'-vK^l by formula and every extn is of this form}
every irreducible character of $U^{l-w}K^{l}$ containing $\psi_{\beta}$
is of the form $\psi_{\hat{\beta}'}$, for some lift $\hat{\beta}'=\begin{pmatrix}0 & \hat{\lambda}'\\
(\hat{\lambda}')^{-1}\hat{\Delta}' & 0
\end{pmatrix}\in\sl_{2}(\cO_{r})$ of $\beta$ (note that it is always true that $l-w\geq l'-v$, as
either $w=v$ or $v=l'$; thus $U^{l-w}\subseteq U^{l'-v}$). We may
therefore choose $\hat{\beta}'$ such that $\omega|_{U^{l-w}K^{l}}=\psi_{\hat{\beta}'}$.
Then, by construction,
\[
\tilde{\psi}_{\hat{\beta}}|_{U^{l'}K^{l}}=\omega|_{U^{l'}K^{l}}=\psi_{\hat{\beta}'}|_{U^{l'}K^{l}},
\]
(note that since $\Delta\in\mfp$, we have $w\geq1$, so $l-w\leq l'$
and hence $U^{l'}\subseteq U^{l-w}$) so by Lemma~\ref{lem:every char of U^l'K^l above psi_beta is =00005Cpsi_hat-beta}
we must have $\hat{\lambda}^{-1}\hat{\Delta}\equiv(\hat{\lambda}')^{-1}\hat{\Delta}'\mod{\mfp}^{l}$.
Thus by Lemma~\ref{lem:intersection CJ cap C(g) is ZU^1K^l cap C(g)},
\[
ZU^{l-w}K^{l}\cap C_{G}(g_{\hat{\beta}})=ZU^{l-w}K^{l}\cap C_{G}(g_{\hat{\beta}'}),
\]
so, as $\omega$ is necessarily trivial on $Z$ (since $\rho$ is),
the above claim is equivalent to the claim that $\psi_{\hat{\beta}'}$
contains the trivial character when restricted to $U^{l-w}K^{l}\cap C_{G}(g_{\hat{\beta'}})$.
But since $U^{l-w}\subseteq U^{l'-v}$, this indeed holds by Lemma~\ref{lem: psi_hat-beta res to U^l'-vK^l cap C is the trivial char}.
\end{proof}

\subsection{The main theorem when $p$
is odd}
\label{subsec:The-main-theorem}

Using the theorems in the preceding subsections, it is now easy to
finish the proof of our main theorem when $p$ is odd. The only remaining
ingredient is the following lemma and its $\SL_{2}(\F_{q})$-analogue,
which form the base cases of the inductive procedure of the proof.
\begin{lem}
\label{lem:GL_2 over finite field} Every irreducible
character of $\GL_{2}(\F_{q})$ that is trivial on the centre is contained
in the conjugation character.
\end{lem}

\begin{proof}
We make use of the character table of $\GL_{2}(\F_{q})$, available
for example in \cite[Section~11.7]{Digne-Michel_2nd_ed} to show that
any irreducible character that is trivial on the centre, when restricted
to a suitable centraliser, contains the trivial representation. First
let $\chi$ be a character of degree $q+1$. Then $\chi$ has value
$(q+1)\alpha(a)\beta(a)$ on any scalar conjugacy class $\left(\begin{smallmatrix}a\\
 & a
\end{smallmatrix}\right)$, where $\alpha,\beta\in\Irr(\F_{q}^{\times})$, $\alpha\neq\beta$,
so it is trivial on the centre if and only if $\alpha=\beta^{-1}$.
Let 
\[
C=C_{\GL_{2}(\F_{q})}\left(\begin{smallmatrix}1 & 1\\
0 & 1
\end{smallmatrix}\right)=\left\{ \left(\begin{smallmatrix}a & b\\
0 & a
\end{smallmatrix}\right)\mid a\in\F_{q}^{\times},b\in\F_{q}\right\} .
\]
On any element $\left(\begin{smallmatrix}a & b\\
0 & a
\end{smallmatrix}\right)$ where $b\neq0$, $\chi$ has value $\alpha(a)\beta(a)$. Assume that
$\chi$ is trivial on the centre. Then
\begin{align*}
|C|\cdot\langle\chi|_{C},\mathbf{1}\rangle & =\sum_{g\in C}\chi(g)=\sum_{a\in\F_{q}^{\times}}\chi(\left(\begin{smallmatrix}a\\
 & a
\end{smallmatrix}\right))+\sum_{a,b\in\F_{q}^{\times}}\chi(\left(\begin{smallmatrix}a & b\\
 & a
\end{smallmatrix}\right))\\
 & =(q-1)(q+1)+(q-1)^{2}=2(q-1)q,
\end{align*}
and thus $\langle\chi|_{C},\mathbf{1}\rangle\neq0$. (Note that
if we tried to use the diagonal subgroup $T$ instead of $C$ we would
run into problems for $q=2$, where $T=\{1\}$ is not the centraliser
of any element in $\GL_{2}(\F_{q})$. Alternatively, one could by
note that $\GL_{2}(\F_{2})=\SL_{2}(\F_{2})$, use \cite[Corollary~5]{tiep2023conjugation}
and proceed with $C=T$ and $q\geq3$.)

If $\chi$ is a character of degree $1$, then it has value $\alpha(a^{2})$
on any scalar conjugacy class $\left(\begin{smallmatrix}a\\
 & a
\end{smallmatrix}\right)$, where $\alpha\in\F_{q}^{\times}$, so it is trivial on the centre
if and only if $\alpha(a^{2})=1$ for all $a$. The same computation
as above shows that 
\begin{align*}
|C|\cdot\langle\chi|_{C},\mathbf{1}\rangle & =2(q-1)q,
\end{align*}
so $\langle\chi|_{C},\mathbf{1}\rangle\neq0$.

Next, let $\chi$ be a character of degree $q$. Then then it has
value $q\alpha(a^{2})$ on any scalar conjugacy class $\left(\begin{smallmatrix}a\\
 & a
\end{smallmatrix}\right)$, where $\alpha\in\F_{q}^{\times}$, so as in the previous case, it
is trivial on the centre if and only if $\alpha(a^{2})=1$ for all
$a$. Meanwhile, it has value $0$ on any element $\left(\begin{smallmatrix}a & b\\
0 & a
\end{smallmatrix}\right)$ where $b\neq0$, so if $\chi$ is trivial on the centre, then
\begin{align*}
|C|\cdot\langle\chi|_{C},\mathbf{1}\rangle & =\sum_{a\in\F_{q}^{\times}}\chi(\left(\begin{smallmatrix}a\\
 & a
\end{smallmatrix}\right))=q(q-1),
\end{align*}
so $\langle\chi|_{C},\mathbf{1}\rangle\neq0$.

Finally, let $\chi$ be a character of degree $q-1$. Then $\chi$
has value $(q-1)\omega(a)$ on any scalar conjugacy class $\left(\begin{smallmatrix}a\\
 & a
\end{smallmatrix}\right)$, where $\omega\in\Irr(\F_{q^{2}}^{\times})$, $\omega^{q}\neq\omega$,
so it is trivial on the centre if and only if $\omega(a)=1$ for all
$a\in\F_{q}^{\times}\subset\F_{q^{2}}^{\times}$ . As in \cite[Section~11.7]{Digne-Michel_2nd_ed},
we label non-central conjugacy classes that meet the non-split maximal
torus by elements $\left(\begin{smallmatrix}x & 0\\
0 & x^{q}
\end{smallmatrix}\right)$, where $x\in\F_{q^{2}}^{\times}$ and $x\neq x^{q}$. The element
in $\GL_{2}(\F_{q})$ corresponding to $\left(\begin{smallmatrix}x & 0\\
0 & x^{q}
\end{smallmatrix}\right)$ is $\lambda^{-1}\left(\begin{smallmatrix}x & 0\\
0 & x^{q}
\end{smallmatrix}\right)\lambda$, where $\lambda\in\GL_{2}(\bar{\F}_{q})$ is any element such that
$\lambda^{-1}F(\lambda)=\left(\begin{smallmatrix}0 & 1\\
1 & 0
\end{smallmatrix}\right)$ (here $F$ is the $q$-power Frobenius map). The non-split torus
$T_{w}\subset\GL_{2}(\F_{q})$ is 
\[
T_{w}=\{\lambda^{-1}\left(\begin{smallmatrix}x & 0\\
0 & x^{q}
\end{smallmatrix}\right)\lambda\mid x\in\F_{q^{2}}^{\times}\}=\bigcup_{a\in\F_{q}^{\times}}\left(\begin{smallmatrix}a\\
 & a
\end{smallmatrix}\right)\cup\bigcup_{\substack{x\in\F_{q^{2}}^{\times}\\
x\neq x^{q}
}
}\lambda^{-1}\left(\begin{smallmatrix}x & 0\\
0 & x^{q}
\end{smallmatrix}\right)\lambda
\]
Take any $x\in\F_{q^{2}}^{\times}$ such that $x\neq x^{q}$ (note
that such an $x$ always exists since $\F_{q}=\{x\in\F_{q^{2}}\mid x^{q}=x\}$
and $|\F_{q^{2}}^{\times}\setminus\F_{q}^{\times}|\geq(4-1)-(2-1)=2$).
Then
\[
T_{w}=C_{\GL_{2}(\F_{q})}\lambda^{-1}\left(\begin{smallmatrix}x & 0\\
0 & x^{q}
\end{smallmatrix}\right)\lambda
\]
 since $\left(\begin{smallmatrix}x & 0\\
0 & x^{q}
\end{smallmatrix}\right)$ is regular and semisimple. On any non-central element of $T_{w}$,
$\chi$ has value $-(\omega(x)+\omega(x^{q}))$. Assume that $\chi$
is trivial on the centre. Then
\begin{align*}
|T_{w}|\cdot\langle\chi|_{C},\mathbf{1}\rangle & =\sum_{a\in\F_{q}^{\times}}\chi(\left(\begin{smallmatrix}a\\
 & a
\end{smallmatrix}\right))+\sum_{\substack{x\in\F_{q^{2}}^{\times}\\
x\neq x^{q}
}
}\chi(\left(\begin{smallmatrix}x & 0\\
0 & x^{q}
\end{smallmatrix}\right))=(q-1)^{2}-\sum_{\substack{x\in\F_{q^{2}}^{\times}\\
x\neq x^{q}
}
}\omega(x)+\omega(x^{q})\\
 & =(q-1)^{2}-\left(-\sum_{a\in\F_{q}^{\times}}2\omega(x)\right)\text{ (since }\text{\ensuremath{\sum_{x\in\F_{q^{2}}^{\times}}\omega(x)=0}}\text{)}\\
 & =(q-1)^{2}+2(q-1),
\end{align*}
and thus $\langle\chi|_{T_{w}},\mathbf{1}\rangle\neq0$.
\end{proof}

\begin{thm}
\label{thm: main thm when p odd}Assume
that $p\neq2$ and let $G=\GL_{2}(\cO_{r})$ or $G=\SL_{2}(\cO_{r})$.
Then any $\rho\in\Irr(G)$ that is trivial on the centre is contained
in the conjugation character of $G$.
\end{thm}

\begin{proof}
We prove this by induction on $r$. When $r=1$, the result is Lemma~\ref{lem:GL_2 over finite field}  for $\GL_2(\F_q)$ and for $\SL_2(\F_q)$ it follows from the general results in  \cite{Heide-Saxl-Tiep-Zalesski}, together with the verification for $\SL_{2}(\F_{2})$ and $\SL_{2}(\F_{3})$ in \cite[Corollary~5]{tiep2023conjugation}. 

Now let $r\geq2$ and assume that the result holds for $G_{r-1}$. Let $\rho\in\Irr(G)$
be trivial on the centre. If $\rho$ is not twist-primitive, the result
follows from Lemma~\ref{lem:reducing to twist-primitive when p odd}
and the induction hypothesis. If $\rho$ is twist-primitive, then
$\rho$ is regular by Lemma~\ref{lem:constr_regular_chars_for_GL2_or_SL2}\,(\ref{enu: constr lemma i)})
and the result follows from Theorems~\ref{thm:p odd - r even},
\ref{thm:p odd - r odd - v < l'}
and \ref{thm:p even - r odd - v=00003Dl'}.
\end{proof}

\section{\protect\label{sec:The-case-GL2,v=00003D2,p=00003D2}The case when
$\mathbf{G}=\mathrm{GL}_2$, $v=l'$ and $p=2$}

For this case, we exploit the fact that when $\bfG=\GL_{2}$ and $p=2$,
the group $Z^{l'}UK^{l}$ is actually normal in $UK^{l'}$ . This
does not hold when $\bfG=\SL_{2}$ or when $p\neq2$.
\begin{lem}
\label{lem: GL_2, p=00003D2, normal and abelian quotient}Assume
that $\bfG=\GL_{2}$ and $p=2$. Then $[U,K^{l'}]\subseteq Z^{l'}UK^{l}$.
Thus $Z^{l'}UK^{l}$ is normal in $UK^{l'}$ and $UK^{l'}/Z^{l'}UK^{l}$
is an elementary abelian $p$-group.
\end{lem}

\begin{proof}
Let $u\in U$ and $k\in K^{l'}$. By direct matrix computation,
$$
kuk^{-1} \in 
\left\{ \begin{pmatrix}1-x\pi^{l'} & 0\\
0 & 1+x\pi^{l'}
\end{pmatrix}\mid x\in\cO_{r}\right\} UK^{l}.
$$
For any $x\in\cO_{r}$, we have

\[
\begin{pmatrix}1-x\pi^{l'} & 0\\
0 & 1+x\pi^{l'}
\end{pmatrix}\begin{pmatrix}1+x\pi^{l'} & 0\\
0 & 1+x\pi^{l'}
\end{pmatrix}\in K^{l},
\]
since $2x\in\mfp$, so 
\[
\left\{ \begin{pmatrix}1-x\pi^{l'} & 0\\
0 & 1+x\pi^{l'}
\end{pmatrix}\mid x\in\cO_{r}\right\} UK^{l}\subseteq Z^{l'}UK^{l}
\]
 and thus $[U,K^{l'}]\subseteq Z^{l'}UK^{l}$ (note that this step
does not work for $\SL_{2}$, as $(1+x\pi^{l'})^{2}=1$ fails in general.).
Since $U$ is abelian and $K^{l}$ is normal in $G$, this implies
that $Z^{l'}UK^{l}$ is normal in $UK^{l'}$.

For the quotient, we have
\[
^ {}UK^{l'}/Z^{l'}UK^{l}=\frac{K^{l'}(Z^{l'}UK^{l})}{Z^{l'}UK^{l}}\cong\frac{K^{l'}}{Z^{l'}U^{l'}K^{l}},
\]
which is a quotient of $K^{l'}/K^{l}\cong\M_{2}(\F_{q})$ and is therefore
an elementary abelian $p$-group.
\end{proof}
\begin{lem}
\label{lem:Stab-computation}Assume that $\bfG=\GL_{2}$, $p=2$ and
$v=l'$. Let $\beta=\begin{pmatrix}0 & \lambda\\
\lambda^{-1}\Delta & 0
\end{pmatrix}\in\sl_{2}(\cO_{l'})$ with $\lambda\in\cO_{l'}^{\times}$ and Let $\hat{\beta}=\begin{pmatrix}0 & \hat{\lambda}\\
\hat{\lambda}^{-1}\hat{\Delta} & 0
\end{pmatrix}\in\sl_{2}(\cO_{r})$ be a lift of $\beta$ and $\psi_{\hat{\beta}}$ the corresponding
character of $U^{l'-v}K^{l}$ (as introduced in Proposition~\ref{prop:There exist extns to U^l'-vK^l by formula and every extn is of this form}).
Let $\delta\in\cO_{r}$ be such that $\hat{\Delta}=\delta\pi^{l'}$.
Then, for any extension $\psi_{\hat{\beta}}'$ of $\psi_{\hat{\beta}}$
to $Z^{l'}UK^{^{l}}$, we have 
\[
\Stab_{UK^{l'}}(\psi_{\hat{\beta}}')=\begin{cases}
UJ & \text{if\,}\hat{\Delta}\in\mfp^{l},\\
Z^{l'}U\left\{ \begin{pmatrix}1+\delta^{-1}(\hat{\lambda}c)^{2}\pi^{l'} & 0\\
c\pi^{l'} & 1
\end{pmatrix}\mid c\in\cO_{r}\right\} K^{l} & \text{if\,}\hat{\Delta}\in\mfp^{l'}\setminus\mfp^{l}.
\end{cases}
\]
\end{lem}

\begin{proof}
Write $S:=\Stab_{UK^{l'}}(\psi_{\hat{\beta}}')$. Any element of $UK^{l'}$
can be written as 
\[
\begin{pmatrix}1 & u\\
0 & 1
\end{pmatrix}\begin{pmatrix}1+x\pi^{l'} & 0\\
0 & 1+x\pi^{l'}
\end{pmatrix}\begin{pmatrix}1+a\pi^{l'} & 0\\
c\pi^{l'} & 1
\end{pmatrix},
\]
for some $u,x,a,c\in\cO_{r}$ and since $UZ^{l'}\subset S$, it is
enough to determine when $\begin{pmatrix}1+a\pi^{l'} & 0\\
c\pi^{l'} & 1
\end{pmatrix}\in S$. Furthermore, as $K^{l'}$ centralises $Z^{l'}K^{l}$, it is enough
to determine when $\begin{pmatrix}1+a\pi^{l'} & 0\\
c\pi^{l'} & 1
\end{pmatrix}$ stabilises $\psi_{\hat{\beta}}'|_{U}$. We have
\begin{align*}
 & \begin{pmatrix}1+a\pi^{l'} & 0\\
c\pi^{l'} & 1
\end{pmatrix}\begin{pmatrix}1 & u\\
0 & 1
\end{pmatrix}\begin{pmatrix}(1+a\pi^{l'})^{-1} & 0\\
-c(1+a\pi^{l'})^{-1}\pi^{l'} & 1
\end{pmatrix}\\
 & =\begin{pmatrix}1+a\pi^{l'} & u(1+a\pi^{l'})\\
c\pi^{l'} & 1+uc\pi^{l'}
\end{pmatrix}\begin{pmatrix}(1+a\pi^{l'})^{-1} & 0\\
-c(1+a\pi^{l'})^{-1}\pi^{l'} & 1
\end{pmatrix}\\
 & \in\begin{pmatrix}* & u+ua\pi^{l'}\\
-uc^{2}\pi^{r-1} & *
\end{pmatrix}
\end{align*}
and 
\[
\psi_{\hat{\beta}}'\begin{pmatrix}* & u+ua\pi^{l'}\\
-uc^{2}\pi^{r-1} & *
\end{pmatrix}=\psi(-\hat{\lambda}uc^{2}\pi^{r-1}+\hat{\lambda}^{-1}\hat{\Delta}(u+ua\pi^{l'})).
\]
We thus have $\begin{pmatrix}1+a\pi^{l'} & 0\\
c\pi^{l'} & 1
\end{pmatrix}\in S$ if and only if 
\[
\psi(-\hat{\lambda}uc^{2})\pi^{r-1}+\hat{\lambda}^{-1}\hat{\Delta}(u+ua\pi^{l'}))=\psi(\hat{\lambda}^{-1}\hat{\Delta}u),\qquad\text{for all }u\in\cO_{r}.
\]
This is equivalent to 
\[
\psi(u[-\hat{\lambda}c^{2}+\hat{\lambda}^{-1}\delta a]\pi^{r-1})=1,\qquad\text{for all }u\in\cO_{r},
\]
that is, to
\begin{equation}
(\hat{\lambda}c)^{2}\equiv\delta a\mod{\mfp}.\label{eq:Stab-condition}
\end{equation}
If $\hat{\Delta}\in\mfp^{l}$, then $\delta\in\mfp$ so (\ref{eq:Stab-condition})
is equivalent to $c\in\mfp$ and the stabiliser $S$ in this case
is therefore
\[
Z^{l'}U\begin{pmatrix}1+\mfp^{l'} & 0\\
0 & 1
\end{pmatrix}K^{l}=UJ.
\]
On the other hand, if $\hat{\Delta}\in\mfp^{l'}\setminus\mfp^{l}$,
then $\delta\in\cO_{r}^{\times}$ and (\ref{eq:Stab-condition}) is
equivalent to 
\[
a\equiv\delta^{-1}(\hat{\lambda}c)^{2}\mod{\mfp}
\]
and
\[
\begin{pmatrix}1+a\pi^{l'} & 0\\
c\pi^{l'} & 1
\end{pmatrix}=\begin{pmatrix}1+\delta^{-1}(\hat{\lambda}c)^{2}\pi^{l'}+x\pi^{l} & 0\\
c\pi^{l'} & 1
\end{pmatrix}
\]
for some $x\in\cO_{r}$. But since there exists a $y\in\cO_{r}$ such
that
\[
\begin{pmatrix}1+\delta^{-1}(\hat{\lambda}c)^{2}\pi^{l'}+x\pi^{l} & 0\\
c\pi^{l'} & 1
\end{pmatrix}=\begin{pmatrix}1+\delta^{-1}(\hat{\lambda}c)^{2}\pi^{l'} & 0\\
c\pi^{l'} & 1
\end{pmatrix}\begin{pmatrix}1+y\pi^{l} & 0\\
0 & 1
\end{pmatrix}
\]
we conclude that the stabiliser $S$ in this case is 
\[
Z^{l'}U\left\{ \begin{pmatrix}1+\delta^{-1}(\hat{\lambda}c)^{2}\pi^{l'} & 0\\
c\pi^{l'} & 1
\end{pmatrix}\mid c\in\cO_{r}\right\} K^{l}.
\]
\end{proof}
\begin{lem}
\label{lem:intersection Stab cap C(g) is ZUK^l cap C(g) for GL_2, v=00003Dl'}Let
the hypotheses and notation be as in Lemma~(\ref{lem:Stab-computation}).
Moreover, in case $\hat{\Delta}\in\mfp^{l'}\setminus\mfp^{l}$, assume
that $\cO_{1}=\F_{2}$. Then
\[
\Stab_{UK^{l'}}(\psi_{\hat{\beta}}')\cap C_{G}(g_{\hat{\beta}})=Z^{l'}UK^{l}\cap C_{G}(g_{\hat{\beta}}).
\]
\end{lem}

\begin{proof}
Suppose first that $\hat{\Delta}\in\mfp^{l}$ and let $h\in C_{G}(g_{\hat{\beta}})$.
Then 
\begin{align*}
h_{l} & \in C_{G_{l}}\Big(\begin{smallmatrix}1 & \hat{\lambda}_{l}\\
-\hat{\lambda}_{l}^{-1}\hat{\Delta}_{l} & 1-\hat{\Delta}_{l}
\end{smallmatrix}\Big)=C_{G_{l}}\Big(\begin{smallmatrix}1 & \hat{\lambda}_{l}\\
0 & 1
\end{smallmatrix}\Big)=Z_{l}^ {}U_{l}
\end{align*}
so that $h\in ZUK^{l}$. But $\Stab_{UK^{l'}}(\psi_{\hat{\beta}}')\cap ZUK^{l}\subseteq UK^{l'}\cap ZUK^{l}=Z^{l'}UK^{l}$,
so $\Stab_{UK^{l'}}(\psi_{\hat{\beta}}')\cap C_{G}(g_{\hat{\beta}})\subseteq Z^{l'}UK^{l}$.
This trivially implies $\Stab_{UK^{l'}}(\psi_{\hat{\beta}}')\cap C_{G}(g_{\hat{\beta}})\subseteq Z^{l'}UK^{l}\cap C_{G}(g_{\hat{\beta}})$,
but it is clear that $Z^{l'}UK^{l}\subseteq\Stab_{UK^{l'}}(\psi_{\hat{\beta}}')$
so also $Z^{l'}UK^{l}\cap C_{G}(g_{\hat{\beta}})\subseteq\Stab_{UK^{l'}}(\psi_{\hat{\beta}}')\cap C_{G}(g_{\hat{\beta}})$,
which proves the desired equality.

Suppose now that $\hat{\Delta}\in\mfp^{l'}\setminus\mfp^{l}$ so that
\[
\Stab_{UK^{l'}}(\psi_{\hat{\beta}}')=Z^{l'}U\left\{ \begin{pmatrix}1+\delta^{-1}(\hat{\lambda}c)^{2}\pi^{l'} & 0\\
c\pi^{l'} & 1
\end{pmatrix}\mid c\in\cO_{r}\right\} K^{l}
\]
 by Lemma~\ref{lem:Stab-computation}. Let $h\in\Stab_{UK^{l'}}(\psi_{\hat{\beta}}')\cap C_{G}(g_{\hat{\beta}})$.
Then
\begin{align*}
h_{l} & \in Z_{l}^{l'}U_{l}\left\{ \begin{pmatrix}1+\delta_{l}^{-1}(\hat{\lambda}_{l}c)^{2}\pi^{l'} & 0\\
c\pi^{l'} & 1
\end{pmatrix}\mid c\in\cO_{l}\right\} \\
 & \cap\left\{ \begin{pmatrix}x & y\hat{\lambda}_{l}\\
-y\hat{\lambda}_{l}^{-1}\hat{\Delta}_{l} & x-y\hat{\Delta}_{l}
\end{pmatrix}\mid x,y\in\cO_{l}\right\}.
\end{align*}
A matrix in this intersection is of the form
\begin{align*}
\begin{pmatrix}z & 0\\
0 & z
\end{pmatrix}\begin{pmatrix}1 & u\\
0 & 1
\end{pmatrix}\begin{pmatrix}1+\delta_{l}^{-1}(\hat{\lambda}_{l}c)^{2}\pi^{l'} & 0\\
c\pi^{l'} & 1
\end{pmatrix} & =\begin{pmatrix}z(1+(\delta_{l}^{-1}(\hat{\lambda}_{l}c)^{2}+uc)\pi^{l'}) & zu\\
zc\pi^{l'} & z
\end{pmatrix}\\
 & =\begin{pmatrix}x & y\hat{\lambda}_{l}\\
-y\hat{\lambda}_{l}^{-1}\hat{\Delta}_{l} & x-y\hat{\Delta}_{l}
\end{pmatrix},
\end{align*}
where $z\in1+\mfp^{l'}$ and $u\in\cO_{r}$. Comparing the diagonals,
this implies that
\[
(\delta_{l}^{-1}(\hat{\lambda}_{l}c)^{2}+uc)\pi^{l'}=-\hat{\lambda}_{l}c\pi^{l'},
\]
that is, 
\[
\delta_{l}^{-1}(\hat{\lambda}_{l}c)^{2}+uc\equiv-\hat{\lambda}_{l}c\mod{\mfp}.
\]
Assume that $c\not\in\mfp$. Then this is equivalent to 
\[
\delta_{l}^{-1}\hat{\lambda}_{l}c+u\hat{\lambda}_{l}^{-1}\equiv-1\mod{\mfp}.
\]
As $\cO/\mfp=\F_{2}$ and $c\not\in\mfp$, this is equivalent to $u\equiv1\mod{\mfp}$.
But $zc\pi^{l'}=-y\hat{\lambda}_{l}^{-1}\hat{\Delta}_{l}$ and $\hat{\Delta}\in\mfp^{l'}\setminus\mfp^{l}$,
so we must have $y\not\in\mfp$, hence $y\hat{\lambda}_{l}=zu$ implies
that $u\not\in\mfp$; contradiction. Thus $c\in\mfp$ so that $zc\pi^{l'}=0$
and
\[
h_{l}\in\left\{ \begin{pmatrix}x & y\\
0 & x
\end{pmatrix}\mid x,y\in\cO_{l}\right\} \cap G_{l},
\]
whence $h\in ZUK^{l}$. As in the first case, this implies $\Stab_{UK^{l'}}(\psi_{\hat{\beta}}')\cap C_{G}(g_{\hat{\beta}})\subseteq Z^{l'}UK^{l}$
and hence the desired equality.
\end{proof}
\begin{lem}
\label{lem: Heisenberg GL_2, p=00003D2, v=00003Dl'}
Let the hypotheses and notation be as in Lemma~\ref{lem:Stab-computation}
and let $S=\Stab_{UK^{l'}}(\psi_{\hat{\beta}}')$. Let $R$ be the
subgroup of $S$ whose image in $S/Z^{l'}UK^{l}$ is the radical of
the bilinear form defined by $(xZ^{l'}UK^{l},yZ^{l'}UK^{l})\mapsto\psi_{\hat{\beta}}([x,y])$.
Then, for any $\sigma\in\Irr(UK^{l'}\mid\psi_{\hat{\beta}})$ there
exists an extension $\theta$ of $\psi_{\hat{\beta}}'$ to $R$ such
that $\Ind_{R}^{UK^{l'}}\theta$ is a multiple of $\sigma$.
\end{lem}

\begin{proof}
By Lemma~\ref{lem: GL_2, p=00003D2, normal and abelian quotient}
we can apply \cite[Corollary~3.3]{Stasinski-Stevens} to conclude
there exists an extension $\theta$ of $\psi_{\hat{\beta}}'$ to $R$
such that $\Ind_{R}^{S}\theta$ is a multiple of an irreducible character
$\eta_{\theta}\in\Irr(S\mid\psi_{\hat{\beta}}')$. By standard Clifford
theory, $\Ind_{S}^{UK^{l'}}\eta_{\theta}$ is irreducible, so $\Ind_{R}^{UK^{l'}}\theta$
is a multiple of $\sigma$.
\end{proof}
\begin{thm}
\label{thm: GL_2, p even, v=00003Dl'}Assume
that $G=\GL_{2}(\cO_{r})$, $r$ is odd, $v=l'$ and $\cO_{1}=\F_{2}$.
Then any regular $\rho\in\Irr(G)$ that is trivial on $Z$ is contained
in the conjugation character of $G$.
\end{thm}

\begin{proof}
By Lemma~\ref{lem: tau can be taken as 0},
$\rho\in\Irr(G\mid\psi_{\beta})$ for some $\beta=\begin{pmatrix}0 & \lambda\\
\lambda^{-1}\Delta & 0
\end{pmatrix}\in\mfg(\cO_{l}),\lambda\in\cO_{l}^{\times}$. By Lemma~\ref{lem:constr_regular_chars_for_GL2_or_SL2},
$\rho=\Ind_{C_{G}(\hat{\beta})K^{l'}}^{G_{}}\sigma$ for some $\sigma\in\Irr(C_{G}(\hat{\beta})K^{l'}\mid\psi_{\beta})$
and any choice of lift $\hat{\beta}$. By Lemma~\ref{lem: is contained in the conj rep iff sigma }\,(\ref{enu: i)})
with $H=G$ and $N=C_{G}(\hat{\beta})K^{l'}$, it is enough to prove
that $\sigma$ restricted to $C_{G}(\hat{\beta})K^{l'}\cap C_{G}(g)$
contains the trivial character, for some $g\in G$. 

For any lift $\hat{\beta}=\begin{pmatrix}0 & \hat{\lambda}\\
\hat{\lambda}^{-1}\hat{\Delta} & 0
\end{pmatrix}$ of $\beta$, Lemma~\ref{lem:intersection of centralisers in ZUK}
says that,
\[
C_{G}(\hat{\beta})K^{l'}\cap C_{G}(g_{\hat{\beta}})=ZUK^{l'}\cap C_{G}(g_{\hat{\beta}}).
\]
It is thus sufficient to show the following.
\begin{claim*}
There exists a lift $\hat{\beta}=\begin{pmatrix}0 & \hat{\lambda}\\
\hat{\lambda}^{-1}\hat{\Delta} & 0
\end{pmatrix}$ such that $\sigma$ contains the trivial character when restricted
to $ZUK^{l'}\cap C_{G}(g_{\hat{\beta}})$.
\end{claim*}
As $\rho$ is trivial on $Z$, so is $\sigma$ (by Lemma~\ref{lem: rho trivial on Z iff sigma trivial on Z}),
so it is enough to show the claim for $\sigma$ restricted to $UK^{l'}\cap C_{G}(g_{\hat{\beta}})$.
By Lemma~\ref{lem: Heisenberg GL_2, p=00003D2, v=00003Dl'},
for any irreducible constituent $\sigma'$ of $\sigma|_{UK^{l'}}$
there exists a $\psi_{\hat{\beta}}\in\Irr(UK^{l'}\mid\psi_{\beta})$,
for some lift $\hat{\beta}=\begin{pmatrix}0 & \hat{\lambda}\\
\hat{\lambda}^{-1}\hat{\Delta} & 0
\end{pmatrix}$, an arbitrary extension $\psi_{\hat{\beta}}'$ of $\psi_{\hat{\beta}}$
to $Z^{l'}UK^{l'}$ and an extension $\theta$ of $\psi_{\hat{\beta}}'$
to a subgroup $R$ of $S:=\Stab_{UK^{l'}}(\psi_{\hat{\beta}}')$ such
that $\Ind_{R}^{UK^{l'}}\theta$ is a multiple of $\sigma'$. Thus,
by Lemma~\ref{lem: is contained in the conj rep iff sigma }\,(\ref{enu: ii)})
with $H=UK^{l'}$, $N=R$ and $C=C_{G}(g_{\hat{\beta}})$, the above
claim holds if the following claim holds.
\begin{claim*}
$\theta$ restricted to $R\cap C_{G}(g_{\hat{\beta}})$ is trivial.
\end{claim*}
But, as $R\subseteq S$,
\[
R\cap C_{G}(g_{\hat{\beta}})\subseteq S\cap C_{G}(g_{\hat{\beta}})=Z^{l'}UK^{l}\cap C_{G}(g_{\hat{\beta}}),
\]
where the equality comes from Lemma~\ref{lem:intersection Stab cap C(g) is ZUK^l cap C(g) for GL_2, v=00003Dl'}
(which applies since we have assumed that $\cO_{1}=\F_{2}$), so it
is enough to prove that $\theta$ restricted to $Z^{l'}UK^{l}\cap C_{G}(g_{\hat{\beta}})$
is trivial.

Since $\theta|_{UK^{l}}=\psi_{\hat{\beta}}$ and $\theta$ is necessarily
trivial on $Z^{l'}$, we are reduced to showing that $\psi_{\hat{\beta}}$
contains the trivial character when restricted to 
\[
U^ {}K^{l}\cap C_{G}(g_{\hat{\beta}}).
\]
This last statement indeed holds by Lemma~\ref{lem: psi_hat-beta res to U^l'-vK^l cap C is the trivial char},
so the above claims follow and the proof is complete.
\end{proof}

\section{The case $\mathrm{SL}_2(\mathcal{O}_r)$ with $p=2$}
\label{sec:The case SL_2, p is 2}
In this section we solve the Hain--Tiep problem in the last remaining case, $p=2$, which we assume throughout this section. We adopt the following notation:
for any integer $i\geq 1$, let $S_i:= \SL_{2}(\cO_{i})$, $G_i:= \GL_{2}(\cO_{i})$ and let $Z_i$ be the centre of $S_i$. Moreover, let $G=G_r$, $S=S_r$ and $Z=Z_r$. We also let $Z_{G}$ denote the centre of $G$.
While it is well-known that $\rho_{r-1}$ maps $Z_{G}$  surjectively onto $Z_{G_{r-1}}$, this is not true for $Z$ and $Z_{r-1}$. 

Let $e$ denote the absolute ramification index of $\cO_r$, that is, the integer such that $(2):=2\cO_r=\mfp^e$.
The following lemma determines $Z$ explicitly when $e=1$ and implies that $\rho_{r-1}(Z)=\pm I$, where $I$ is the identity matrix.
\begin{lem}
\label{lem:Centre p=2}Assume
that $e=1$. Then 
\[
Z=(\pm1+\mfp^{r-1})I.
\]
\end{lem}
\begin{proof}
 We determine $Z\cong\{a\in\cO_{r}\mid a^{2}=1\}$
explicitly. If $a\in\cO_{r}$ is such that $a^{2}=1$, then $a=1+2x$,
for some $x\in\cO_{r}$ and
\[
(1+2x)^{2}=1+2^{2}x+2^{2}x^{2}=1.
\]
This is equivalent to
\begin{equation}
2^{2}x(1+x)=0.\label{eq: condition for centre-1}
\end{equation}
If $x$ is a unit and $1+x$ is a unit, then (\ref{eq: condition for centre-1})
holds if and only if $r\leq2$. If $x$ is a unit and $1+x$ is not
a unit, then (\ref{eq: condition for centre-1})
holds if and only if $1+x\in(2)^{r-2}$, that is, $x=-1+2^{r-2}y$,
for some $y\in\cO_{r}$, so that
\[
1+2x=1+2(-1+2^{r-2}y)=-1+2^{r-1}y.
\]
On the other hand, if $x$ is not a unit, then $1+x$ is a unit, and
(\ref{eq: condition for centre-1}) holds if and
only if $x\in(2)^{r-2}$, that is, $x=2^{r-2}y$ for some $y\in\cO_{r}$,
so that
\[
1+2x=1+2^{r-1}y.
\]
In summary, we have shown that 
\[
Z=\left\{ \begin{pmatrix}\pm1+2^{r-1}y & 0\\
0 & \pm1+2^{r-1}y
\end{pmatrix}\mid y\in\cO_{r}\right\} = (\pm1+\mfp^{r-1})I.
\]
\end{proof}
Recall from Section~\ref{sec:Basics-of-the-chars} the characters $\psi_{[\beta]}\in \Irr(\KSL^l)$, where $\beta\in\M_{2}(\cO_{l'})$ and $[\beta]$ denotes the coset of $\beta$ modulo scalar matrices. 

\begin{lem}
\label{lem: tau can be taken as 0-1}
Let
$\beta=\begin{pmatrix}\tau & \lambda\\
\lambda^{-1}\Delta & 0
\end{pmatrix}\in\M_{2}(\cO_{l'}),\lambda\in\cO_{l'}^{\times}$ and assume that $l'\geq e$. Let $\chi\in\Irr(S\mid\psi_{[\beta]})$
be a representation that is trivial on $\KSL^l\cap Z$.  Then $\chi\in\Irr(S\mid\psi_{[\beta']})$ for some $\beta'$ of the form
\[
\begin{pmatrix}0 & \lambda'\\
(\lambda')^{-1}\Delta' & 0
\end{pmatrix},
\]
with $\lambda'\in \cO_{l'}^{\times}$.
\end{lem}
\begin{proof}
By Clifford's theorem, it is enough to prove that there exists
an element in $[\beta]$ that is $\SL_{2}(\cO_{l'})$-conjugate
to an element with zero diagonal.
Since $\chi$ is trivial on $\KSL^{l}\cap Z$,  so is $\psi_{[\beta]}$.

As $l'\geq e$, we have $r-e\geq l$ and moreover,  $(1+\pi^{r-e}a)^{2}\in1+\mfp^{e}\mfp^{r-e}+\mfp^{2(r-e)}=1$
when $l\geq e$ (which holds, as $l\geq l'$). Therefore
$$\KSL^{l}\cap Z\supseteq \left\{ \begin{pmatrix}1+\pi^{r-e}a & 0\\
0 & 1+\pi^{r-e}a
\end{pmatrix}\mid a\in\cO_{r}\right\}$$ 
and thus
\[
\psi_{[\beta]}\begin{pmatrix}1+\pi^{r-e}a & 0\\
0 & 1+\pi^{r-e}a
\end{pmatrix}=\psi(\hat{\tau}\pi^{r-e}a)=1,
\]
for all $a$; hence $\hat{\tau}\in\mfp^{e}$ and so $\tau\in\mfp^{e}$.
For any $x\in\cO_{l'}$,
\[
\beta+xI=\begin{pmatrix}\tau+x & \lambda\\
\lambda^{-1}\Delta & x
\end{pmatrix}
\]
is regular as an element of $\M_{2}(\cO_{l'})$ and has trace $\tau+2x$
and determinant $-\Delta+x(\tau+x)$ and is therefore $\SL_{2}(\cO_{l'})$-conjugate
to a matrix of the form
\[
\begin{pmatrix}\tau+2x & \lambda'\\
(\lambda')^{-1}\Delta' & 0
\end{pmatrix},
\]
where $\Delta':=\Delta-x(\tau+x)$ and $\lambda'\in\cO_{l'}^{\times}$.
Writing $2=\pi^{e}u$ for some unit $u\in\cO_{r}^{\times}$ and $\tau=\pi^{e}v$,
for some $v\in\cO_{r}$, we can choose $x=-u^{-1}v$, so that $\tau+2x=0$. 
\end{proof}

Let $\hat{\chi}\in\Irr(S_{n})$
be a representation that is trivial on $Z_{n}$ with $K_n^r$ the maximal subgroup among the $K_n^i$ contained in the kernel of $\hat{\chi}$.
Thus, $\hat{\chi}$ is equal to the inflation $\Inf^{S_n}_{S}(\chi)$ for some \emph{primitive} $\chi\in\Irr(S)$. Let
$$\Zimage := \rho_{n,r}(Z_n)= \begin{cases}
Z & \text{if }n=r,\\
\pm I & \text{if }n>r,
\end{cases}$$
where the case when $n>r$ follows from \cref{lem:Centre p=2}.
As $\hat{\chi}$ is assumed to be trivial on $Z_n$ and on $K^r$, it factors through 
$$\frac{S_{n}}{Z_{n}K_{n}^{r}}\cong\frac{S_{n}/K_{n}^{r}}{Z_{n}K_{n}^{r}/K_{n}^{r}}\cong S/\Zimage,$$
so the character $\chi$ it defines on $S$ is trivial on $\Zimage$.

For $g\in S_n$ we will write
$$\Cimage := \rho_{n,r}(C_{S_n}(g)),$$
which should not cause confusion as $n$ and $r$ are fixed (but arbitrary, with $r\geq 1$), throughout.
\begin{lem}
\label{inflation}
Let $\hat{\chi} = \Inf^{S_n}_S(\chi)$ for some character $\chi$ of $S$. Then
the restriction of $\hat{\chi}$  to the centraliser $C_{S_n}(g)$ of some $g\in S_n$ contains the trivial character if and only if $\chi$ restricted to $\Cimage$ contains the trivial character.
\end{lem}

\begin{proof}
  Note that $\Res_{C_{S_n}(g)}^{S_n} \Inf^{S_n}_{S}(\chi)=\Inf^{C_{S_n}(g)}_{\Cimage}\Res_{\Cimage}^{S} (\chi)$.  Moreover,
    \[
\langle\Inf^{C_{S_n}(g)}_{\Cimage}\Res_{\Cimage}^{S} (\chi),\mathbf{1}\rangle=\langle\Res_{C_{S_n}(g)}^{S} (\chi),\mathbf{1}\rangle,
\]
as inflation is adjoint to the $C \cap K^r$-fixed space construction. 
\end{proof}

Thus, we will work to show that there is a $g\in S_n$  such that $\chi$ restricted to
$\Cimage$  contains the trivial character. As $\chi$ is primitive, it is regular (see \cref{lem:constr_regular_chars_for_GL2_or_SL2}), that is,  $\chi \in \Irr(S\mid \psi_{[\beta]})$  with $\beta=\begin{pmatrix}\tau & \lambda\\
\lambda^{-1}\Delta & 0
\end{pmatrix}\in\M_{2}(\cO_{l'})$, with $\lambda \in \cO_{l'}^{\times}$.

\subsection{Special group elements and intersections with their
centralisers}



\begin{lem}
\label{lem:conjugation}
Let $\alpha\in\M_{2}(\cO_{l'})$ be regular. Assume that $e=1$ and $\cO_{1}=\F_2$. There exists an $x \in \cO_{l'}$ and $\lambda\in \cO_{l'}^{\times}$ such that $\alpha + xI$ is $\SL_{2}(\cO_{l'})$-conjugate to $\beta$ of one of the following forms. 
\begin{enumerate}

\item $\beta=\begin{pmatrix} 0 & \lambda\\
\lambda^{-1}\Delta & 0 
\end{pmatrix}$;
    \item $\beta=\begin{pmatrix}\tau & \lambda\\
\lambda^{-1}\Delta & 0
\end{pmatrix}$ such that $\tau + \lambda^{-1}\Delta=2\lambda$;

    \item $\beta=\begin{pmatrix}\tau & \lambda\\
0 & 0
\end{pmatrix}$.

\end{enumerate}

\end{lem}

\begin{proof}
Since $\alpha + xI$ is regular, it is $\SL_{2}(\cO_{l'})$-conjugate
to a matrix of the form
\[
\beta=\begin{pmatrix}\tau & \lambda\\
\lambda^{-1}\Delta & 0
\end{pmatrix},
\]
where $\lambda\in \cO_{l'}^{\times}$,  $\tau=\Tr(\alpha)+2x$ and $\Delta=-(\det(\alpha)+x(\Tr(\alpha)+x))$.

For \textit{Case 1}, assume that $\Tr(\alpha)$ is not a unit. Since $e=1$ we can write $2=\pi u$ for some unit $u\in\cO_{l'}^{\times}$ and when $\Tr(\alpha)$ is not a unit, then $\tau$ is not a unit, so we can write $\tau=\pi v$ 
for some $v\in\cO_{l'}$. We can choose $x=-u^{-1}v$, so that $\tau+2x=0$. Thus, we can obtain $\beta=\begin{pmatrix} 0 & \lambda\\
\lambda^{-1}\Delta & 0 
\end{pmatrix}$ when $\Tr(\alpha)$ is not a unit. 

For \textit{Case 2}, we assume that $\Tr(\alpha)$ and $\det(\alpha)$ are units. Define the function
\begin{align*}
  F: \cO_{l'}^{\times} & \longrightarrow \cO_{l'} \setminus \cO_{l'}^{\times},\\
 F(x) & =  \Tr(\alpha)+2x - \lambda^{-1}(\det(\alpha)+x(\Tr(\alpha)+x))  =  \tau + \lambda^{-1}\Delta,
\end{align*}
Note that $F(x)$ is always a non-unit since $\Tr(\alpha)$ and $\det(\alpha)$ are units and $\cO_1=\F_2$. We claim that $F$ is bijective. Indeed, if  $F(x)=F(y)$, then  $(y-x)(\lambda^{-1}\Tr(\alpha)-2+\lambda^{-1}(y+x))=0$. As $x$ and $y$ are units, we have that $x+y$ is not a unit and therefore, since $\Tr(\alpha)$ is a unit, $\lambda^{-1}\Tr(\alpha)-2+\lambda^{-1}(y+x)$ is a unit. Thus $y-x=0$, so $F$ is injective. Since $\cO_1=\F_2$, we have  
$$|\cO_{l'}^{\times}|=|1+\mfp|=|\mfp|=|\cO_{l'} \setminus \cO_{l'}^{\times}|,$$ hence $F$ is also surjective. Therefore, there is an $x \in \cO_{l'}^{\times}$ such that $\tau + \lambda^{-1}\Delta=2\lambda$. 

For \textit{Case 3}, assume that $\Tr(\alpha)$ is a unit and $\det(\alpha)$ is not a unit. 
%
%
We claim that the function 
\begin{align*}
  F: \cO_{l'}^{\times} & \longrightarrow \cO_{l'} \setminus \cO_{l'}^{\times},\\
 F(x) & =x(\Tr(\alpha)+ x),
\end{align*}
is injective and therefore surjective. Indeed, if $F(x) = F(y)$, then $(x-y)(\Tr(\alpha)+x+y)=0$. As $\Tr(\alpha)+x+y$ is a unit, we conclude that $x=y$. Thus, there exists an $x$ such that $\Delta = -(\det(\alpha) + x(\Tr(\alpha)+x))=0$ and therefore we can obtain, 
$$\beta=\begin{pmatrix}\tau & \lambda\\
0 & 0
\end{pmatrix}.$$

\end{proof}

\begin{lem}
\label{intercept} 
Assume that either of the following conditions holds.
\begin{enumerate}

\item  $\hat{\beta}=\begin{pmatrix}\hat{\tau} & \hat{\lambda}\\
\hat{\lambda}^{-1}\hat{\Delta} & 0
\end{pmatrix}$, $\hat{\lambda},\hat{\tau},\hat{\Delta}\in \cO_r^{\times}$ and $g=\begin{pmatrix}1 & 1\\
w & 1+w
\end{pmatrix}\in G$ with $w \in \mfp$. 

\item  ${\hat{\beta}}=\begin{pmatrix}\hat{\tau} & \hat{\lambda}\\
0 & 0
\end{pmatrix}$, $\hat{\lambda},\hat{\tau}\in \cO_r^{\times}$ and $g=\begin{pmatrix}1 & 1\\
0 & 1
\end{pmatrix}\in G$. 
\end{enumerate}
Then for any $1\leq i\leq r$,   $$C_{G}({\hat{\beta}})K^{i}\cap C_{G}(g)\subseteq Z_{G}K^{i}.$$ 

\end{lem}

\begin{proof}
Assume (1) and let $h\in C_{G}({\hat{\beta}})K^{i}\cap C_{G}(g)$. Then
\begin{align*}
h_{i} & \in C_{G_{i}}(\hat{\beta}_i)\cap C_{G_{i}}(g_i)\\
 & \subseteq\left\{ \begin{pmatrix}a+b\hat{\tau}_i & b\hat{\lambda}_i\\
b\hat{\lambda}_i^{-1}\hat{\Delta}_i & a
\end{pmatrix}\mid a,b\in\cO_{i}\right\} \cap\left\{ \begin{pmatrix}x & y\\
yw & x+yw
\end{pmatrix}\mid x,y\in\cO_{i}\right\}.
\end{align*}
A matrix in this intersection is of the form
\[
\begin{pmatrix}a+b\hat{\tau}_i & b\hat{\lambda}_i\\
b\hat{\lambda}_i^{-1}\hat{\Delta}_i & a
\end{pmatrix}=\begin{pmatrix}x & y\\
yw & x+yw
\end{pmatrix},
\]
which implies that $b\hat{\tau}_i=-yw=-b\hat{\lambda}_i w$. Since $\tau$ is a unit and $w\in \mfp$, this implies that $b=y=0$. Thus $h_i \in Z_{G_i}$ and since $Z_{G}$ maps surjectively onto $Z_{G_i}$, we get $h\in Z_{G}K^i$.

Assume (2) and let $h \in C_{G}({\hat{\beta}})K^{i}\cap C_{G}(g)$. Then $h_{i}= \begin{pmatrix}x & y\\
 0 & x
\end{pmatrix}$, for some $x\in \cO_{i}^{\times}$, $y\in \cO_{i}$ and  $h_{i}\hat{\beta}_i h_{i}^{-1}=\hat{\beta}_i$. By computation, we have
\[
h_{i}\hat{\beta}_i h_{i}^{-1}-\hat{\beta}_i=\begin{pmatrix}0 &  -yx^{-1}\hat{\tau}_i\\
0 & 0
\end{pmatrix},
\]
which implies that $y=0.$ Thus, $h\in Z_G K^{i}$.
\end{proof}

In the following, we will say that $\beta \in \M_2(\cO_{l'})$ is of \emph{type n}, with $n=1,2,3$, if it is as in case ($n$) in \cref{lem:conjugation}.

\begin{lem}
\label{tau and delta units}
Assume that $\cO_{1}=\F_2$.
Let $\beta \in\M_{2}(\cO_{l'})$ and assume that one of the following holds.
\begin{enumerate}
    \item $\beta$ is of type 2 and $g=\begin{pmatrix}1 & 1\\
w & 1+w
\end{pmatrix}\in S_n$ with $w=-2+\pi^l \in \cO_n$ and $n > r$;
    \item $\beta$ is of type 3 and $g=\begin{pmatrix}1 & 1\\
0 & 1
\end{pmatrix}\in S_n$.
\end{enumerate}
Let $\psi'_{\beta} \in \Irr(Z_{G} K^{l})$ be an extension of $\psi_{\beta}$ that is trivial on $\pm I$. Then $\psi'_{\beta}$ is trivial on $C_{G}(\hat{\beta})K^l\cap \Cimage$, for every lift $\hat{\beta}$.

\end{lem}

\begin{proof}

Assume that we are in the first case (1).
Let $h \in C_{S_n}(g)$ such that $h_r$ is in $C_{G}(\hat{\beta})K^l\cap \Cimage$. Then as $h \in C_{S_n}(g)$, we have
\begin{equation*}
h=\begin{pmatrix}x+ y &   y\\
yw & x+y(1+w)
\end{pmatrix}
\end{equation*} with $x$, $y$ in $\cO_{n}$.

By the first part of \cref{intercept}, $h_r \in Z_{G}K^{l}$, so we conclude that $y \in \mfp^l$. Thus, by a change of variable we have: 

\begin{equation*}
h=\begin{pmatrix}x+\pi^{l} y &  \pi^{l} y\\
\pi^{l} yw & x+\pi^{l} y(1+w)
\end{pmatrix}.
\end{equation*}

As $w=-2+\pi^l$, we have by computation, $$1=\det(h)\equiv x^2+\pi^{2l} (xy+y^2(1+w)) \mod{\mfp^{r+1}}$$ (the congruence makes sense as $n>r$). Note that since $x\equiv 1 \mod{\mfp}$, we have
\begin{align*} 
    xy+y^2(1+w)\equiv y+ y^2 \equiv 0 \mod{\mfp},
\end{align*}
where the last congruence holds because $\cO_1=\F_2$.
We conclude that $x^2 \equiv 1 \mod{\mfp^{r+1}}$ independently of $r$ being even or odd. By \cref{lem:Centre p=2}, we thus have that $x \equiv \pm 1 \mod{\mfp^{r}}$. By a change of variables, we can write



\begin{equation*}
         h_r=\pm I \cdot \begin{pmatrix}1+\pi^l y & \pi^l y\\
\pi^l yw_r & 1+\pi^l y(1+w_r)
\end{pmatrix}
\end{equation*} with $y \in \cO_r$.

 As $\psi'_{\beta}$ is trivial on $\pm I$, we have $$\psi'_{\beta}(h_r)=\psi_{\beta}\begin{pmatrix}1+\pi^l y & \pi^l y\\
\pi^l yw_r & 1+\pi^l y(1+w_r))\end{pmatrix} 
= \psi(\pi^l y (\hat{\tau} +\hat{\lambda} w_r + \hat{\lambda}^{-1}\hat{\Delta})),$$
for any lifts $\hat{\tau},\hat{\lambda},\hat{\Delta} \in \cO_r$

As $\tau + \lambda^{-1}\Delta=2\lambda$ and $w_r=-2+\pi^l\in \cO_r$, we have:
\begin{equation*}
    \psi(\pi^l y (\hat{\tau} +\hat{\lambda} w_r + \hat{\lambda}^{-1}\hat{\Delta}))=\psi(\pi^l y (2\lambda-2\lambda))=1.
\end{equation*}
Thus, we conclude that $\psi'_{\beta}$ is trivial on $C_{G}(\hat{\beta})K^l\cap \Cimage$.

Assume now that we are in the second case (2).  Any $h$ in $C_{G}(\hat{\beta})K^l\cap \Cimage$ has the form $h=\begin{pmatrix} x & \pi^l y\\
0 & x
\end{pmatrix}$  with $x^2= 1 $, by \cref{intercept}. 
Since $\psi'_{\beta}$ is trivial on $\pm I$, we have  $$\psi'_{\beta}(h)=\psi_{\beta}\begin{pmatrix}1&  x^{-1}\pi^l y\\
0 & 1\end{pmatrix}=\psi(0)=1.$$ Thus $\psi_{\hat{\beta}}$ is trivial on $C_{G}(\hat{\beta})K^l\cap \Cimage$.

\end{proof}

\subsection{Lemmas about triviality of characters on subgroups}
In this section we prove two technical lemmas that form the main bulk of the proof of our main theorem for $\SL_2(\Z/2^n)$.

\begin{lem}
\label{lem triv on centre triv on intersection}
Let $n>r$ and $\beta$ be of type 1. Let $\theta \in \Irr(Z_{G}U^{l'-v}K^l \mid \psi_{\beta})$ and assume that $\theta$ is trivial on $ \pm I(Z_{G}\cap K^{l-1})$. 
Then there exists a $g \in S_n$ such that for any lift $\hat{\beta}$ the following three properties hold.
\begin{enumerate}
    \item
    \label{enu:theta trivial}
    $\theta$ restricted to $Z_{G}U^{l'-v}K^l\cap \Cimage$ is the trivial character;
\item \label{enu:first incl}
$C_{G}(\hat{\beta})K^l\cap \Cimage \subseteq Z_{G}U^{l'-v}K^{l}$;
 \item \label{enu:second incl} 
   if $v<l'$, then $C_{G}(\hat{\beta})K^{l'}\cap \Cimage \subseteq Z_{G}C^1 K^{l'}$.
\end{enumerate}
\end{lem}
\begin{proof}
By (\ref{eq:U^l'-v is in CK^l'}), $Z_{G}U^{l'-v}$ is an abelian group that stabilises $\psi_{\beta}$; hence $\theta$ is an extension of $\psi_{\beta}$. We claim that for any $2\leq i \leq r$,
\begin{equation}
\label{eq: Z-inclusion} Z_{G}K^{i} \cap S \subseteq \pm I(Z_{G}\cap K^{i-1}) K^{i}. 
\end{equation}
Indeed, let $z\in Z_{G}$ and $k\in K^i$ such that $zk\in S$. Then $z_i \in Z_{S_i}$ so by \cref{lem:Centre p=2}, $z_{i-1} = \pm I$. Thus  $z \in \pm I(Z_{G}\cap K^{i-1})$, which proves the claim.

By \cref{prop:There exist extns to U^l'-vK^l by formula and every extn is of this form},
$\theta$ restricted to $U^{l'-v}K^{l}$ is equal to $\psi_{\hat{\beta}}$, for some lift $\hat{\beta}=\begin{pmatrix}0 & \hat{\lambda}\\
\hat{\lambda}^{-1}\hat{\Delta} & 0
\end{pmatrix}$ of $\beta$.
Let $g\in S_n$ be a lift of $g_{\hat{\beta}}\in S$, where the latter are the elements defined in Section~\ref{sec:A-family-of}. It follows from (\ref{eq: Z-inclusion}) with $i=l$ and the facts that $\Cimage \subseteq S$ and $U^{l'-v} \subseteq S$ that 
$$Z_{G}U^{l'-v}K^{l} \cap \Cimage \subseteq \pm I(Z_{G}\cap K^{l-1})U^{l'-v} K^{l}.$$
For (\ref{enu:theta trivial}) it is therefore sufficient to prove that $\theta$ is trivial when restricted to
$\pm I(Z_{G}\cap K^{l-1})U^{l'-v} K^{l}\cap \Cimage$.
By \cref{lem: psi_hat-beta res to U^l'-vK^l cap C is the trivial char}, $\psi_{\hat{\beta}}$ is trivial on $U^{l'-v}K^{l}\cap C_G(g_{\hat{\beta}})$.
As $\theta$ is trivial on $\pm I(Z_{G}\cap K^{l-1})$, it is thus trivial on $Z_{G}U^{l'-v}K^l\cap \Cimage$. 

Next, recall from just before \cref{lem:intersection of centralisers in ZUK} the definition and properties of the number $w$ attached to the lift $\hat{\beta}$ of $\beta$. We have $l'-v\leq l-w$, so $U^{l-w} \subseteq U^{l'-v}$ and also $\Cimage \subseteq C_G(g_{\hat{\beta}})$. 
Thus, by \cref{lem:intersection of centralisers in ZUK}, 
$$C_G(\hat{\beta})K^l \cap \Cimage \subseteq Z_{G}U^{l-w}K^l \subseteq Z_{G}U^{l'-v}K^l,$$
proving statement (\ref{enu:first incl}). 
Finally, (\ref{enu:second incl})  follows from \cref{lem:intersection of centralisers in ZUK}, using the fact that $U^{l'-v}\subseteq C^1K^{l'}$ when $v<l'$.
Note that if $g$ is fixed, then (\ref{enu:first incl}) and (\ref{enu:second incl}) hold for any choice of $\hat{\beta}$, since the group $C_G(\hat{\beta})K^l$ only depends on $\beta$.
\end{proof}

\begin{lem}
\label{lem p=2 triv on Cimage-type 1}
Let $n>r$ and $\beta$ be of type 1. Let $\rho \in \Irr(G \mid \psi_{\beta})$ and assume  that $\rho$ is trivial on $ \pm I(Z_{G}\cap K^{l'-1})$. 
Then there exists a $g \in S_n$ such that $\rho$ restricted to $\Cimage$ contains the trivial character.
\end{lem}

\begin{proof}
By \cref{lem:constr_regular_chars_for_GL2_or_SL2}, $\rho=\Ind_{C_{G}(\hat{\beta})K^{l'}}^{G}\sigma$ for some $\sigma\in\Irr(C_{G}(\hat{\beta})K^{l'}\mid\psi_{\beta})$, for any lift $\hat{\beta}$ of $\beta$. 
We will prove the following.
\begin{claimABC}
    There exists a $g\in S_n$ such that $\sigma$ restricted to $C_{G}(\hat{\beta})K^{l'}\cap \Cimage$ contains the trivial character.
\end{claimABC}
By \cref{lem: is contained in the conj rep iff sigma } this claim will imply that $\rho$ restricted to $\Cimage$ contains the trivial character.

We will treat the cases when $r$ is even or odd in turn. First, let $r$ be even, so that $l=l'$.
 Let $\theta \in \Irr(Z_{G}U^{l'-v}K^l \mid \psi_{\beta})$ be an irreducible constituent of the restriction of $\sigma$  to $Z_{G}U^{l'-v}K^l $. 

As $\rho$ is trivial on  $ \pm I(Z_{G}\cap K^{l-1})$, so is  $\theta$. Thus, by \cref{lem triv on centre triv on intersection} there exists a $g\in S_n$ such that $\theta$ restricted to $Z_{G}U^{l'-v}K^l\cap \Cimage$ is trivial and such that
 $C_{G}(\hat{\beta})K^l\cap \Cimage \subseteq Z_{G}U^{l'-v}K^{l}$, and therefore, 
$$C_{G}(\hat{\beta})K^l\cap \Cimage \subseteq Z_{G}U^{l'-v}K^l\cap \Cimage,$$
for any lift $\hat{\beta}$. Since $\theta$ restricted to $Z_{G}U^{l'-v}K^l \cap \Cimage $ contains the trivial character, $\sigma$ restricted to $C_{G}(\hat{\beta})K^l\cap \Cimage$ contains the trivial character, proving Claim~A.

Assume now that $r=2l-1$ is odd. We will break down the proof of Claim~A into two cases.

For Case~1, assume that $v<l'$. Let $C^1:=C_G(\hat{\beta})\cap K^1$. By Lemma~\ref{lem:Heisenberg_induction}, there exists an irreducible
constituent of $\sigma|_{Z_{G}C^{1}K^{l'}}$ of the form $\eta_{\theta}$ such that $\Ind_{Z_{G}C^{1}K^{l}}^{Z_{G}C^{1}K^{l'}}\theta$
is a multiple of $\eta_{\theta}$, for some extension $\theta\in\Irr(Z_{G}C^{1}K^{l})$ of $\psi_{\beta}$. 
Let $\theta'$ be the restriction of $\theta$ to $Z_{G}U^{l'-v}K^l$.  Note that $\theta'$ is trivial on $ \pm I(Z_{G}\cap K^{l-1})$, since $\rho$ is and $K^{l-1}\subseteq K^{l'-1}$.
Therefore, by applying \cref{lem triv on centre triv on intersection} to $\theta'$, there exists a $g \in S_n$ such that $\theta' $ restricted to $Z_{G} U^{l'-v}K^{l}\cap \Cimage $ is trivial,  $C_{G}(\hat{\beta})K^l\cap \Cimage \subseteq Z_{G}U^{l'-v}K^{l}$
and  $C_{G}(\hat{\beta})K^{l'}\cap \Cimage \subseteq Z_{G}C^1 K^{l'}$.
The last inclusion implies that in order to prove Claim~A, it is enough to prove the following.
\begin{claimABC}
  The character  $\eta_{\theta}$ restricted to $Z_{G} C^1 K^{l'}\cap \Cimage$ contains the trivial character.  
\end{claimABC}
By \cref{lem: is contained in the conj rep iff sigma } with $H=Z_{G}C^{1}K^{l'}$, $N=Z_{G}C^{1}K^{l}$ and $C=\Cimage$, Claim~B holds if $\theta $ restricted to $Z_{G}C^{1}K^{l}\cap \Cimage$ is trivial.
But $\theta' $ restricted to $Z_{G} U^{l'-v}K^{l}\cap \Cimage$ is trivial, so $\theta $ restricted to $Z_{G}C^{1}K^{l}\cap \Cimage \subseteq C_G(\hat{\beta})K^l \cap \Cimage$ is trivial.
This proves Claim~B and therefore Case~1 of Claim~A. 

For Case 2, 
assume that  $v=l'$, that is,  $\beta= \begin{pmatrix} 0 & \lambda\\
0 & 0
\end{pmatrix}.$
Let $\alpha \in \Irr(\pm I(Z_{G}\cap K^{l'-1})UK^{l'}\mid\psi_{\beta})$ be an irreducible constituent of the restriction of $\sigma$. Then $\alpha$ is trivial on $\pm I(Z_{G}\cap K^{l'-1})$ because $\rho$ is.

Let $\alpha'$ be an irreducible constituent of $\alpha|_{UK^{l'}}$. By Lemma~\ref{lem: Heisenberg GL_2, p=00003D2, v=00003Dl'},
there exists a $\psi_{\hat{\beta}}\in\Irr(UK^{l}\mid\psi_{\beta})$,
for some lift $\hat{\beta}=\begin{pmatrix}0 & \hat{\lambda}\\
\hat{\lambda}^{-1}\hat{\Delta} & 0
\end{pmatrix}$, an extension $\psi_{\hat{\beta}}'$ of $\psi_{\hat{\beta}}$
to $Z^{l'}UK^{l}$ and an extension $\theta$ of $\psi_{\hat{\beta}}'$
to a subgroup $R$ of $\Stab_{UK^{l'}}(\psi_{\hat{\beta}}')$ such
that $\Ind_{R}^{UK^{l'}}\theta$ is a multiple of $\alpha'$.

Let $g\in S_n$ be a lift of $g_{\hat{\beta}}$. Then it follows from \cref{lem:intersection of centralisers in ZUK} that 
$$C_{G}(\hat{\beta})K^{l'}\cap \Cimage \subseteq Z_{G}UK^{l'}.$$
It follows from (\ref{eq: Z-inclusion}) with $i=l'$ and the facts that $\Cimage \subseteq S$ and $U \subseteq S$ that
$$ C_{G}(\hat{\beta})K^{l'} \cap \Cimage \subseteq Z_{G}UK^{l'} \cap S \subseteq \pm I(Z_{G}\cap K^{l'-1})UK^{l'}.$$ 
It follows that $\alpha$ restricted to $\pm I(Z_{G}\cap K^{l'-1})UK^{l'} \cap  \Cimage $ contains the trivial character if the restriction of $\alpha$ to
$$ UK^{l'} \cap \pm I(Z_{G}\cap K^{l'-1})\Cimage \subseteq  UK^{l'} \cap C_G(g_{\hat{\beta}})$$
contains the trivial character. Furthermore, this holds if $\alpha'$  contains the trivial character on $UK^{l'} \cap C_G(g_{\hat{\beta}})$.
Thus, since $\Ind_{R}^{UK^{l'}}\theta$ is a multiple of $\alpha'$, Lemma~\ref{lem: is contained in the conj rep iff sigma }\,(\ref{enu: ii)}) 
with $H=UK^{l'}$, $N=R$ and $C=C_G(g_{\hat{\beta}})$ implies that Claim~A holds if the following holds.
\begin{claimABC}
$\theta$ restricted to $R\cap C_G(g_{\hat{\beta}})$ is trivial.
\end{claimABC}
Since $R\subseteq \Stab_{UK^{l'}}(\psi_{\hat{\beta}}')$,
\[
R\cap C_G(g_{\hat{\beta}}) \subseteq \Stab_{UK^{l'}}(\psi_{\hat{\beta}}')\cap C_G(g_{\hat{\beta}})=Z_G^{l'}UK^{l}\cap C_G(g_{\hat{\beta}}),
\]
where the equality comes from Lemma~\ref{lem:intersection Stab cap C(g) is ZUK^l cap C(g) for GL_2, v=00003Dl'}
(which applies since we have assumed that $\cO_{1}=\F_{2}$), so to prove Claim~C it
is enough to prove that $\theta$ restricted to $Z_G^{l'}UK^{l}\cap C_{G}(g_{\hat{\beta}})$
is trivial.

Since $\theta|_{UK^{l}}=\psi_{\hat{\beta}}$ and $\theta$ is necessarily
trivial on $Z_G^{l'}$ (because $\rho$ is), we are reduced to showing that $\psi_{\hat{\beta}}$
contains the trivial character when restricted to 
\[
U^ {}K^{l}\cap C_G(g_{\hat{\beta}}).
\]
This last statement indeed holds by Lemma~\ref{lem: psi_hat-beta res to U^l'-vK^l cap C is the trivial char},
so this proves Claim~C and therefore Case~2 of Claim~A. 
\end{proof}

\setcounter{claimABC}{0}
We now prove the corresponding lemma for types 2 and 3.
\begin{lem}
\label{lem p=2 triv on Cimage-type 2 and 3}
Let $n>r$ and $\beta$ be of type 2 or 3. Let $\rho \in \Irr(G \mid \psi_{\beta})$ and assume  that $\rho$ is trivial on $ \pm I$. 
Then there exists a $g \in S_n$ such that $\rho$ restricted to $\Cimage$ contains the trivial character.
\end{lem}
\begin{proof}
Just like in the proof of \cref{lem p=2 triv on Cimage-type 1}, $\rho=\Ind_{C_{G}(\hat{\beta})K^{l'}}^{G}\sigma$ for some $\sigma\in\Irr(C_{G}(\hat{\beta})K^{l'}\mid\psi_{\beta})$, for any lift $\hat{\beta}$ of $\beta$, and  
we reduce to proving: 
\begin{claimABC}
    There exists a $g\in S_n$ such that $\sigma$ restricted to $C_{G}(\hat{\beta})K^{l'}\cap \Cimage$ contains the trivial character.
\end{claimABC}
Let $g$ be as in \cref{tau and delta units}\,(1) for $\beta$ of type 2 and as in (2) for $\beta$ of type 3, respectively.
It follows from \cref{intercept} that 
$$C_{G}({\hat{\beta}})K^{i}\cap \Cimage\subseteq Z_{G}K^{i},$$
for $i\in\{l,l'\}$.

First, let $r$ be even, so that $l=l'$.
Let $\psi'_{\beta} \in \Irr(Z_{G}K^l\mid \psi_{\beta})$ be an irreducible component of the restriction of $\sigma$ to $Z_{G}K^l$. Then $\psi'_{\beta}$ is trivial on $\pm I$ since $\rho$ is. Claim~A follows from \cref{tau and delta units}.

Assume now that $r$ is odd. 
By Lemma~\ref{lem:Heisenberg_induction}, there exists an irreducible
constituent of $\sigma|_{Z_{G}C^{1}K^{l'}}$ of the form $\eta_{\theta}$ such that $\Ind_{Z_{G}C^{1}K^{l}}^{Z_{G}C^{1}K^{l'}}\theta$
is a multiple of $\eta_{\theta}$, for some extension $\theta\in\Irr(Z_{G}C^{1}K^{l})$ of $\psi_{\beta}$. 
Let $\psi'_{\beta} \in \Irr(Z_{G}K^l\mid \psi_{\beta})$ be an irreducible component of the restriction of $\theta$ to $Z_{G}K^l$. Then $\psi'_{\beta}$ is trivial on $\pm I$ since $\rho$ is. By the same argument as in Case~1 in the proof of \cref{lem p=2 triv on Cimage-type 1}, the proof of Claim~A is reduced to proving that $\theta$ restricted to $Z_{G}C^{1}K^{l}\cap \Cimage$ is trivial.
But by \cref{tau and delta units}, $\psi'_{\beta} $ restricted to $Z_{G}K^{l}\cap \Cimage$ is trivial, so $\theta $ restricted to $Z_{G}C^{1}K^{l}\cap \Cimage \subseteq C_G(\hat{\beta})K^l \cap \Cimage$ is trivial.
\end{proof}

\subsection{The main theorem for $\mathrm{SL}_2(\mathbb{Z}/2^r)$}
Using the lemmas in the preceding subsections, we can now finish the proof of the main theorem for $\SL_2(\Z/2^n)$.
The reason why we state the theorem below for $\SL_2(\Z/2^n)$ rather than for $S=S_r$ is that with this notation almost all the essential work will be carried out in $G$ and $S$ (as in the previous sections of this paper), keeping the notation lighter than if we had worked with some $G_i$ and $S_i$ for $r>i$.

\begin{lem}
\label{lem: Z_G_cap_K^l'-1_cap_S in pm I K^l}
Assume that $e=1$ and $\cO_{1}=\F_2$. Then $Z_{G} \cap K^{l'-1} \cap S \subseteq \pm I K^l$.  
\end{lem}
\begin{proof}
    Let $h \in Z_{G} \cap K^{l'-1} \cap S$ so that $h= \begin{pmatrix} 1+\pi^{l'-1}x & 0\\
0 & 1+\pi^{l'-1} x
\end{pmatrix}$, for $x\in \cO_r$. Then
$$1=\det(h)=1+2\pi^{l'-1}x+\pi^{2l'-2} x^2,$$
so writing $2=u\pi$ for some unit $u\in \cO_r^{\times}$, we get $$x(u+\pi^{l'-2}x)\equiv 0 \mod{\mfp^{l}}.$$
If $l'>2$, then $u+\pi^{l'-2}x$ is always a unit and thus $x \equiv 0 \mod{\mfp^{l}}$ so $h \in K^l$. 

Next, if  $l'=2$, then the above congruence says $x(u+x)\equiv 0 \mod{\mfp^{l}}.$ If $x$ is not unit then $u+x$ is a unit and so $x \equiv 0 \mod{\mfp^{l}}$ so $h \in K^l$. If $x$ is a unit then $u+x \equiv 0 \mod{\mfp^{l}}$ so
$\pi x\equiv -\pi u = -2 \mod{\mfp^{l}}$. Thus, $1+\pi x\equiv -1 \mod{\mfp^{l}}$ so $h \in \pm I K^l$. 

Finally, if $l'=1$, then $r=2$ or $r=3$ and the equation $\det(h)=1$ says  $x(2+x)=0$. This  has no solution for $x$ a unit, so $x=2 z$ for some $z$. Hence, if $r=2$ then $h \in K^l$. When $r=3$, we have $l=2$ and $1+x=1+2z \equiv \pm 1  \mod{\mfp^{2}}$ because $1+(1+2z)=2(1+z)\equiv 0 \mod{\mfp^2}$ when $z$ is a unit and if $z$ is not a unit, then $1+2z\equiv 1 \mod{\mfp^2}$. We conclude that for all $l'\geq 1$, $h \in \pm I K^l$.
\end{proof}
\begin{thm}
\label{thm:Main for SL_2}
Assume that $e=1$ and $\cO_{1}=\F_2$. Let $\hat{\chi} \in\Irr(\SL_{2}(\cO_{n}))$, $n\geq 1$, be a representation that is trivial on the centre $Z_n$. Then $\hat{\chi}$ is contained in the conjugation character of $\SL_{2}(\cO_{n})$. 
\end{thm}

\begin{proof}
 Let $K^r$ be the maximal subgroup among the $K^i$ contained in the kernel of $\hat{\chi}$, so $\hat{\chi}=\Inf^{S_n}_{S}(\chi)$ for some primitive $\chi\in\Irr(S)$. Since $\hat{\chi}$ is trivial on $Z_n$, it follows that $\chi$ is trivial on $\Zimage$. By \cref{inflation}, it is enough to show that $\chi$ restricted to $\Cimage$ contains the trivial character for some $g \in S_n$. 
 
 First we let $r = 1$.  Then $S=\SL_2(\F_2)$ and $Z=\pm I$. It is easily checked (see \cite[Corollary~5]{tiep2023conjugation}) that since $\chi$ is trivial on $Z$, there exists a $\bar{g} \in S$ such that $\chi$ restricted to $C_S(\bar{g})$ contains the trivial character. Picking $g \in S_n$ such that $\rho_1(g)=\bar{g}$ we have that $\chi$ restricted to $\Cimage \subseteq C_S(\bar{g})$ contains the trivial character (cf.~\cite[Lemma~4]{tiep2023conjugation}),  proving the theorem when $r=1$.
 
 Assume from now on that $r\geq 2$. Let $\beta \in\M_{2}(\cO_{l'})$ be such that $\chi \in \Irr(S \mid \psi_{[\beta]})$. By \cref{lem:constr_regular_chars_for_GL2_or_SL2}, $\beta$ is regular. Thus, by \cref{lem:conjugation}, we can assume that $\beta$ is equal to one of the following matrices: 
\begin{enumerate}

\item $\begin{pmatrix} 0 & \lambda\\
\lambda^{-1}\Delta & 0 
\end{pmatrix}$;
    
\item $\begin{pmatrix}\tau & \lambda\\
\lambda^{-1}\Delta & 0
\end{pmatrix}$ such that $\tau + \lambda^{-1}\Delta=2\lambda$;

\item $\begin{pmatrix}\tau & \lambda\\
0 & 0
\end{pmatrix}$ with $\tau$ unit. 
\end{enumerate}
Note that when $n=r$, \cref{lem: tau can be taken as 0-1} implies that $\beta$ can be taken to be of type 1, since $\chi$ is trivial on $Z = \Zimage$.
\setcounter{claimABC}{0}
\begin{claimABC}
 There exists a $\rho \in \Irr(G \mid \chi)$ such that $\rho \in \Irr(G \mid \psi_{\beta'})$, with $\beta'$ of the same type as $\beta$.
\end{claimABC}
We will break down the proof of Claim~A in two cases where $n=r$ and $n>r$, respectively.
Assume that $n=r$ and that  $\beta$ is chosen of type 1. As $\chi \in \Irr(S \mid \psi_{[\beta]})$ is trivial on $Z$, the injection $S/Z \hookrightarrow G/Z_{G}$ implies that there is a $\rho \in \Irr(G \mid \chi)$ that is trivial on $Z_{G}$ and $\rho \in \Irr(G \mid \psi_{\beta+xI})$. Since $\rho$ is trivial on $Z_{G}$, and $\beta+xI$ is regular, \cref{lem: tau can be taken as 0} implies that $\rho \in \Irr(G\mid \psi_{\beta'})$ with
$\beta'$ of type 1. Thus, we have shown Claim~A for $n=r$. 

Assume now that  $n>r$ and recall that in this case $\Zimage=\pm I$. If $l \neq 1$, note that $\Zimage K^l\cong \Zimage \times K^l$ and when $l =1$, we have $\Zimage \subseteq K^l $. Thus, in either case there exists an extension $\psi'_{\beta}$  of $\psi_{\beta}$ to $\Zimage K^l$ that is trivial on $\Zimage$. We then have
\begin{align*}
  \langle \Res_{\Zimage K^l}\Ind_{S}^{G} \chi, \psi'_{\beta} \rangle 
  & \geq \langle \Ind_{S \cap \Zimage K^l}^{\Zimage K^l}( \Res_{S \cap \Zimage K^l} \chi), \psi'_{\beta} \rangle \\
  &= \langle \Res_{S \cap \Zimage K^l} \chi, \Res_{S \cap \Zimage K^l}\psi'_{\beta} \rangle\\
  &\geq \langle  \Res_{S \cap \Zimage K^l}\psi'_{\beta},  \Res_{S \cap \Zimage K^l}\psi'_{\beta} \rangle =1.
\end{align*}
Thus there exists an irreducible constituent $\rho$ of $\Ind_{S}^{G} \chi$, that is, $\rho\in \Irr(G\mid \chi)$, such that $\Res_{\Zimage K^l}\rho$ contains $\psi'_{\beta}$ (note that $\beta'$ can thus be taken to be just $\beta$ in this case). We conclude that Claim~A is now established in all cases.

\begin{claimABC}
 Assume we  have $\beta$ of type 1 and $n>r$. There exists a $\rho \in \Irr(G \mid \chi)$ such that $\rho \in \Irr(G \mid \psi_{\beta'})$, with $\beta'$ of type 1, and $\rho$ is trivial on $\pm I (Z_{G} \cap K^{l'-1})$. 
\end{claimABC}
We first show that $\chi$ is trivial on $\pm I (Z_{G} \cap K^{l'-1} \cap S)$.
It follows from \cref{lem: Z_G_cap_K^l'-1_cap_S in pm I K^l} that $\pm I (Z_{G}\cap K^{l'-1}\cap S) \subseteq \pm IK^l$ and thus
$$\pm I (Z_{G} \cap K^{l'-1} \cap S) \subseteq \pm I  (Z_{G} \cap K^l).$$
By the form of $\beta$, $\psi_{\beta}$ is trivial on $Z_{G} \cap K^l$, so $\chi$ is trivial on $\pm I (Z_{G} \cap K^{l'-1} \cap S)$, as $\chi$ restricted to $Z_{G} \cap K^{l'-1} \cap S$ is a multiple of $\psi_{\beta}$ restricted to this group.

As $\chi \in \Irr(S \mid \psi_{[\beta]})$ is trivial on $\pm I(Z_{G} \cap K^{l'-1} \cap S)$, the injection 
$$S/[\pm I(Z_{G} \cap K^{l'-1} \cap S)]\hookrightarrow G/[\pm I(Z_{G} \cap K^{l'-1})]$$
implies that there is a $\rho \in \Irr(G \mid \psi_{\beta+xI})$, for some $x\in \cO_{l'}$, lying above $\chi$, such that $\rho$  is trivial on $\pm I(Z_{G} \cap K^{l'-1})$.
As $\rho$ is trivial on $Z_{G} \cap K^{l} \subseteq Z_{G} \cap K^{l'-1}$, by \cref{lem: tau can be taken as 0}  we have that $\rho$ contains $\psi_{\beta'}$ for some  $\beta'$ of type 1. Thus Claim B follows.

We will now show that $\chi$ contains the trivial character when restricted to $\Cim(g')$ for some $g' \in S_n$. Let $\rho \in \Irr(G \mid \chi)$ be such that $\rho \in \Irr(G \mid \psi_{\beta})$  with $\beta$ one of the three types (such a $\rho$ exists by Claim~A). Since $\chi$ is trivial on $\Zimage$, so is $\rho$.
Note that when $n=r$ we can assume by the proof of Claim~A that  $\rho$ is trivial on $Z_{G}$. Thus, for $n=r$, where we have $\beta$ of type 1 and $\Cimage=C_G(g)$,
Theorems \ref{thm:p odd - r even},  \ref{thm:p odd - r odd - v < l'} and  \ref{thm:p even - r odd - v=00003Dl'} yield a $g \in S_n$ such that $\rho$ restricted to $\Cimage$ contains the trivial character.
If $\beta$ is of type 1 and $n>r$, we can moreover let $\rho$ be as in Claim~B. 
Then, for $n>r$, Lemmas \ref{lem p=2 triv on Cimage-type 1} and \ref{lem p=2 triv on Cimage-type 2 and 3} yield a $g \in S_n$ such that $\rho$ restricted to $\Cimage$ contains the trivial character.
 Thus, for all $n\geq r$ and all types, $\langle \rho, \Ind_{\Cimage}^{G}\mathbf{1}\rangle \neq 0$ and therefore  $\langle \Ind_{S}^{G}\chi, \Ind_{\Cimage}^{G}\mathbf{1}\rangle \neq 0$. Hence, by the Mackey intertwining number formula,
\[
0 \neq \langle \chi,\Res_{S} \Ind_{\Cimage}^{G}\mathbf{1}\rangle
=\sum_{h\in S \backslash G/\Cimage}\langle \chi,\Ind_{S \cap \leftexp{h}{\Cimage}}^{S}\mathbf{1}\rangle.
\]
Thus, as $\Cimage \subseteq S$, there is an $h \in G$ such that $$\langle \chi,\Ind_{ \leftexp{h}{\Cimage}}^{S}\mathbf{1}\rangle \neq 0.$$
Let $\Hat{h}$ be a lift of $h$ to $G_n$. We have $$S \cap\leftexp{h}{\Cimage}=\leftexp{h}{\Cimage}= \rho_{n,r}(\leftexp{\Hat{h}}{C_{S_n}(g)})=\Cim(\leftexp{\Hat{h}}g).$$

Let $g':=\leftexp{\Hat{h}}g \in S_n$. Then
\[
0 \neq \langle \chi,\Ind_{\Cim(g')}^{S}\mathbf{1}\rangle = \langle \chi|_{\Cim(g')},\mathbf{1}\rangle.
\]
Thus $\chi$ contains the trivial character when restricted to $\Cim(g')$ and the proof is complete.
\end{proof}

\bibliographystyle{alex}
\bibliography{alex}

\end{document}